\documentclass[a4paper,12pt]{article}
\usepackage[utf8]{inputenc}

\usepackage{amsmath}
\usepackage{amssymb}
\usepackage{amsthm}

\usepackage{cite}

\usepackage{booktabs}

\usepackage{esint}

\usepackage{mathrsfs}
\usepackage{bbm}

\usepackage{accents}

\usepackage{color}
\makeatletter
\let\my@saved@original@eqref\eqref % Die Originaldefinition von \eqref sichern.
\renewcommand*{\eqref}[1]{% \eqref umdefinieren:
  \begingroup% Eine Gruppe öffnen, um alle Änderungen lokal zu dieser zu halten.
    \let\normalfont\relax% \normalfont mit \relax gleich setzen und so deaktivieren.
    \my@saved@original@eqref{#1}% Die Originaldefinition von \eqref verwenden.
  \endgroup% Gruppe schließen und damit die Änderung an \normalfont aufheben.
}
\makeatother
\makeatletter
\renewcommand*\env@matrix[1][r]{\hskip -\arraycolsep
  \let\@ifnextchar\new@ifnextchar
  \array{*\c@MaxMatrixCols #1}}
\makeatother

\newcommand{\e}{\mathrm{e}}

\newcommand{\pt}{\partial}
\DeclareMathOperator{\Div}{{div}}
\DeclareMathOperator{\Grad}{{grad}}
\DeclareMathOperator{\Curl}{{curl}}
\newcommand{\jump}[1]{[\hspace*{-2pt}[#1]\hspace*{-2pt}]}

\newcommand{\norm}[2]{\|{#1}\|_{#2}}

\newcommand{\N}{\mathbb{N}}
\renewcommand{\P}{\mathbb{P}}
\newcommand{\R}{\mathbb{R}}

\newcommand{\I}{\mathcal{I}}
\newcommand{\J}{\mathcal{J}}

\newcommand{\lrarrow}{\quad\Leftrightarrow\quad}

\newcommand{\qmbox}[1]{\quad\mbox{#1}\quad}
\newcommand{\pmtrx}[1]{\ensuremath{\begin{pmatrix}#1 \end{pmatrix}}}
\newcommand{\spmtrx}[1]{\ensuremath{\left(\begin{smallmatrix}#1 \end{smallmatrix}\right)}}

\definecolor{seb}{rgb}{0.9,0,0}

\newcommand{\vecsymb}[1]{\boldsymbol{#1}}
\newcommand{\vq}{\vecsymb{q}}
\newcommand{\vx}{\vecsymb{x}}
\newcommand{\dt}{\,\mathrm{d}t}

\usepackage{multirow}
\title{Post-processing and improved error estimates of numerical methods for evolutionary systems}
\author{Sebastian Franz\footnote{
          Institute of Scientific Computing, Technische Universit\"at Dresden, Germany.
          \mbox{e-mail}: sebastian.franz@tu-dresden.de}
       }
\date{\today}
\usepackage{tikz}
\usepackage{pgfplots}
\pgfplotsset{compat=1.9}
\usetikzlibrary{calc}
\newcommand{\scp}[2][H]{\langle #2 \rangle_{#1}}
\newcommand{\scpr}[1]{\langle #1 \rangle_\rho}
\newcommand{\scprm}[1]{\langle #1 \rangle_{\rho,m}}

\newcommand{\K}{\mathcal{K}}
\newcommand{\PS}{\mathbbm{P}}
\newcommand{\U}[1][q]{\PS_{#1}^{disc}(V_k)}
\newcommand{\V}[1][q+1]{\PS_{#1}^{cont}(V_k)}
\newcommand{\tmi}{t_{m,i}}
\newcommand{\Qm}[2][m]{Q_{#1}\left[#2\right]}
\newcommand{\Qmr}[2][\rho]{Q_m\left[#2\right]_{#1}}

\newcommand{\id}[1]{\mathbbm{1}_{#1}}
\theoremstyle{plain}
\newtheorem{thm}{Theorem}[section]
\newtheorem{lem}[thm]{Lemma}

\newtheorem{rem}[thm]{Remark}

\begin{document}
  \maketitle
  \begin{abstract}
    We consider evolutionary systems, i.e. systems of linear partial differential equations arising from
    the mathematical physics. For these systems there exists a general solution theory in exponentially weighted spaces 
    which can be exploited in the analysis of numerical methods. The numerical method considered in this
    paper is a discontinuous Galerkin method in time combined with a conforming Galerkin method in space. Building on
    our recent paper \cite{FrTW16}, we improve some of the results, study the dependence of the numerical 
    solution on the weight-parameter, consider a reformulation and post-processing of its numerical solution.
    As a by-product we provide error estimates for the dG-C0 method.
    Numerical simulations support the theoretical findings.
  \end{abstract}


     
  \section{Introduction}
  Let us consider the problem in the space-time cylinder $\R\times\Omega$, $\Omega\subset\R^d$, given by
  \begin{equation}\label{eq:problem}
    (\pt_t M_0+M_1+A)U=F
  \end{equation}
  where $M_0$, $M_1$ are bounded linear self-adjoint operators on a Hilbert space $H=H(\Omega)$ and $A$ 
  is a skew self-adjoint operator on $H$. Suppose further that there are constants $\rho_0$ and $\gamma>0$, such that
  \[
    \rho M_0+M_1\geq \gamma
  \]
  for all $\rho\geq\rho_0$. Then \cite[Solution Theory]{Picard} states, that the given problem has a 
  unique solution $U$ in
  \[
    H_\rho(\R;H):= 
      \left\{ 
          f:\R\to H\,:\, f \mbox{ meas.}, 
          \int_\R |f(t)|_H^2 \exp(-2\rho t) \dt<\infty
      \right\}
  \]
  for all $F\in H_\rho(\R;H)$. This is a Hilbert space and its natural inner 
  product is given for all $f,g\in H_\rho(\R;H)$ by 
  \[
      \scpr{f,g} :=
      \int_{\R}
      \scp{f(t),g(t)}
      \exp(-2\rho t) \dt.
  \]
  Then we can define the derivative in time $\partial_{t}$ as the closure 
  of the operator
  \[
    \partial_t :C_c^\infty(\R;H)\subseteq H_\rho(\R;H)\to H_\rho(\R;H): 
      \phi\mapsto \phi',
  \]
  where $C_c^\infty(\R;H)$ is the space of infinitely differentiable $H$-valued functions on $\R$ 
  with compact support. We denote the domain of $\partial^k_t$ by $H_\rho^k(\R;H)$ for $k\in \N$.
  The theory, including all the proofs, is also very well presented in the undergraduate-textbook \cite{McGPTW20}
  where we also find the following theorem.
  \begin{thm}[\!\!\mbox{\cite[Theorem 1.3.6]{McGPTW20}}]\label{thm:exist}
    Let $\rho_1,\rho_2\geq\rho_0$ and $F\in H_{\rho_1}(\R;H)\cap H_{\rho_2}(\R;H)$. 
    Denoting the solutions to \eqref{eq:problem} by $U_{\rho_1}\in H_{\rho_1}(\R;H)$ and $U_{\rho_2}\in H_{\rho_2}(\R;H)$, 
    it holds
    \[
      U_{\rho_1}=U_{\rho_2}.
    \]
    Thus, they are independent of $\rho\geq \rho_0$.
  \end{thm}
  
  If we restrict the problem in time to $[0,T]$ by setting $F(t)=0$ for $t<0$ and in space to a finite $\Omega\subset\R^d$,
  we can discretise it.
  In \cite{FrTW16} the above problems are discretised and analysed numerically using a discontinuous Galerkin method in 
  time, see also \cite{AM04,Thomee06}, and an $A$-conforming Galerkin FEM in space. It was shown there, that the error converges in a discretised, 
  exponentially weighted, $\rho$-dependent norm with optimal order in time and space.
  Later, in \cite{Fr19}, a continuous Galerkin method was used for the discretisation in time.

  In this paper we want to extend the above known results on the discretisation. 
  Before starting with the analysis, we give in Section~\ref{sec:basics} the formulation of the numerical method
  and its convergence results, quoting from \cite{FrTW16} and improving it in some aspects.
  In Section~\ref{sec:deprho} we investigate the dependence of the numerical solutions on $\rho$.
  We find that unlike the exact solution, the numerical solutions do depend on $\rho$.
  
  One difficulty in the direct implementation of the methods is the use of the weighted $L^2$-space.
  In Section~\ref{sec:reform} we reformulate the given problem in order to solve it in an unweighted space
  and discretise it. Conveniently, the convergence of the numerical method used in this setting follows 
  from the general setting and we compare the approximations of these two approaches. 
  
  Finally, in Section~\ref{sec:post} we apply a post-processing in time to the numerical solution of 
  the previous section. Its result is the unique solution of another method and we provide convergence error 
  estimates and with an improved order of convergence in time.

  \section{Numerical method and convergence results}\label{sec:basics}
  
  Let us discretise the time interval $[0,T]$ into a set of intervals $I_m=(t_{m-1},t_m]$
  using $0=t_0<t_1<\dots<t_M=T$ of width $\tau_m=t_m-t_{m-1}$, $m\in\{1,\dots,M\}$. For our estimates we will use the 
  maximum length $\tau=\max_{m}\tau_m$. In principle the choice of the mesh is arbitrary. For the 
  results in Section~\ref{sec:post} we will need the condition
  \[
    \frac{\tau_m}{\tau_{m-1}}\leq C,\,1<m\leq M,
  \]
  which limits the growth of the mesh.
  
  In space we also discretise $\Omega\subset \R^d$ into a set $\Omega_h$ of pairwise disjoint 
  simplices of maximum diameter $h$ covering $\Omega$. On this mesh we now define a piecewise 
  polynomial space. We will use $A$-conforming elements, thus the discrete space and the operators in $A$
  must match. There are three basic operators we will consider here, but the concept is general 
  and can be adapted to other operators and associated discrete spaces.
  
  Suppose the operator $A_j$ is a part of $A$ and is applied in $\eqref{eq:problem}$ to a component $U_i\in D(A_j)$
  of $U$. We now need to define a discrete subspace of $D(A_j)$.
  Suppose we have
  \begin{itemize}
    \item $A_j=\Grad$ or $A_j=\Grad^\circ$, i.e. the gradient or gradient with homogeneous (Dirichlet) boundary conditions.
          In this case we need a discrete subset of $H^1(\Omega)$ and we choose 
          globally continuous Lagrange-elements $\PS_k$ which are piecewise polynomials of maximum degree $k$.
          In the case of homogeneous boundary conditions, we also include them into in the discrete space.
          Let $I$ be a classical interpolation operator or a quasi-interpolator of Scott-Zhang type \cite{SZ90} into $\PS_k$.
          Then we have 
          \[
            \norm{U_i-IU_i}{L^2(\Omega)}+\norm{\Grad(U_i-IU_i)}{L^2(\Omega)}
              \leq C h^k
          \]
          assuming $U_i\in H^{k+1}(\Omega)$.
    \item $A_j=\Div$ or $A_j=\Div^\circ$, i.e. the divergence or divergence with homogeneous (Neumann) boundary conditions.
          Here our discrete space is a subset of $H(\Div,\Omega)$ and we use Raviart-Thomas-elements $RT_{k-1}=\PS_{k-1}^d+\vx\PS_{k-1}$
          which are vector valued, piecewise polynomials of degree $k$. Using the standard degrees of freedom we obtain globally
          normal-continuous functions. We can incorporate the homogeneous boundary conditions into the discrete space.
          Let us again denote by $I$ the standard interpolation operator into $RT_{k-1}$. Then it holds
          \[
            \norm{U_i-IU_i}{L^2(\Omega)}+\norm{\Div(U_i-IU_i)}{L^2(\Omega)}
              \leq C h^{k}
          \]
          assuming $U_i\in H^{k}(\Omega)$ such that $\Div U_i\in H^{k}(\Omega)$.
    \item $A_j=\Curl$ or $A_j=\Curl^\circ$, i.e. the curl-operator without or with homogeneous boundary conditions.
          Now the discrete space is a subset of $H(\Curl,\Omega)$ and we use N\'{e}d\'{e}lec elements (of the first kind) 
          $N_{k-1}$, defined as
          \[
            N_{k-1}=\PS_{k-1}^d\oplus S_{k},\quad
            S_{k}=\{\vq\in(\PS_{k}^H)^d\colon \vq\cdot\vx=0\}
          \]
          where $\PS_k^H$ are the homogeneous polynomials of degree $k$. Again, we have vector valued, piecewise polynomials
          of degree $k$ and, using the classical degrees of freedom for these elements, we obtain
          globally tangential-continuous functions. For their interpolation error, again denoting by $I$ the standard interpolation operator,
          it holds
          \[
            \norm{U_i-IU_i}{L^2(\Omega)}+\norm{\Curl(U_i-IU_i)}{L^2(\Omega)}
              \leq C h^k
          \]
          assuming $U_i\in H^{k}(\Omega)$ such that $\Curl U_i\in H^{k}(\Omega)$.
  \end{itemize}
  The above interpolation error estimates are classical results and can be found in many textbooks, e.g.~\cite{EG21a}.
  In this sense we choose for all components of $U$ the corresponding conforming finite element space
  and obtain a composite space $V_k(\Omega)$. Again abusing the notation with $I$ for the interpolation operator into $V_k(\Omega)$,
  we have
  \begin{gather}
    \norm{U-IU}{H}+\norm{A(U-IU)}{H}
      \leq C h^k\label{eq:interspace}
  \end{gather}
  if $U\in H^k(\Omega)$ and $AU\in H^k(\Omega)$. With this setting, the Hilbert space $H$ is 
  \[
    H=\bigotimes_i L^2(\Omega)^{d_i} 
  \]
  with the appropriate dimensions $d_i$ for the component $U_i$.
  Our fully discrete space can now be defined by
  \[
    \U%\PS_q^{disc}(V_k) 
      :=\left\{
          U\in H_\rho(\R,H):
          U|_{I_m}\in\PS_q(I_m,V_k(\Omega)),
          m\in\{1,\dots,M\}
       \right\}
  \]
  which does not have to be continuous in time.
  
  The final component of our numerical method is the numerical approximation of the scalar product in time.
  To do this, we use on each $I_m$ the weighted right-sided Gauß-Radau quadrature rule
  \[
    \Qm{v}_\rho := \frac{\tau_m}{2} \sum_{i=0}^q {\omega}^m_i v(\tmi)
  \]
  whose discrete weights $\omega^m_i$ and points $\tmi$ depend on $\rho$ and are defined such
  that for all polynomials of degree at most $2q$ we have
  \[
    \Qm{p}_\rho = \scprm{p,1}:=\int_{I_m}\scp{p(t),1}\exp(-2\rho(t-t_{m-1}))\dt.
  \]
  
  Let us further introduce the discretised scalar product
  \[
    \Qmr{a,b}:=\Qm{\scp{a,b}}_\rho
  \]
  as an approximation for the scalar product $\scprm{a,b}$. Note that we have with a discrete Cauchy-Schwarz inequality
  \begin{gather}\label{eq:discr:CSU}
    \Qmr{a,b}\leq \norm{a}{Q,\rho,m}\norm{b}{Q,\rho,m}
  \end{gather}
  with the discrete norms 
  \[
    \norm{v}{Q,\rho,m}^2:=\Qmr{v,v}
    \quad\text{and}\quad
    \norm{v}{Q,\rho}^2:=\sum_{m=1}^M\Qmr{v,v}\e^{-2\rho t_{m-1}}
  \]
  as approximations of $\norm{v}{\rho,m}^2:=\scprm{v,v}$ %=\int_{I_m} |v(t)|_H^2 \exp(-2\rho(t-t_{m-1}))\dt$ 
  and $\norm{v}{\rho}^2$. For $v\in\U$ 
  the approximation is exact.
  
  We can now state the discrete \textbf{quadrature formulation}:
  
  For given $F\in H_\rho(\R,H)$ and $x_0\in H$, find $U^\tau_h\in\U$, such that for 
  all $\Phi\in \U$ and $m\in\{1,2,\dots,M\}$ it holds
  \begin{equation}\label{eq:dG}
    \Qmr{(\partial_t M_0+M_1+A)U^\tau_h,\Phi}
      +\scp{M_0 \jump{U^\tau_h}_{m-1}^{x_0},{\Phi}^+_{m-1}}
      =\Qmr{ F,\Phi }.
  \end{equation}
  Here, we denote by 
  \[
    \jump{U^\tau_h}_{m-1}^{x_0}:=
    \begin{cases}
      U^\tau_h(t_{m-1}^+)-U^\tau_h(t_{m-1}^-),& m\in\{2,\ldots,M\}\\ 
      U^\tau_h(t_0^+)-x_0,& m=1,
    \end{cases} 
  \]
  the jump at $t_{m-1}$ and by $\Phi^+_{m-1}:= \Phi(t_{m-1}^+)$.  
  In \cite{FrTW16} we find the following convergence result. Note that the conditions on $M_0$ imply the existence of $M_0^{1/2}$.
  \begin{thm}[\!\!\mbox{\cite[Theorem 4.7]{FrTW16}}]\label{thm:cts}
    We assume for the solution $U$ the 
    regularity     
    \[
      U\in H_\rho^{1}(\R;H^k(\Omega))\cap 
      H_\rho^{q+3}(\R;L^2(\Omega)) 
    \]
    as well as 
    \[
      AU\in H_\rho(\R; H^k(\Omega)).
    \]
    Then we have for the error of the numerical solution
    \[
%       \sup_{t\in[0,T]}\scp{M_0(U-U_h^\tau)(t),(U-U_h^\tau)(t)}+
      \sup_{t\in[0,T]}\norm{M_0^{1/2}(U-U_h^\tau)(t)}{H}+
      \e^{\rho T}\norm{U-U_h^\tau}{Q,\rho}
      \leq C \e^{\rho T}(T^{1/2}\tau^{q+1} + h^{k}).
    \]
  \end{thm}
  \begin{rem}
    The above theorem is given in \cite{FrTW16} in the setting of $A=\spmtrx{0&\Div\\\Grad^\circ&0}$.
    But, as stated there, it holds in the general setting presented here by using $A$-conforming
    finite elements in space.
  \end{rem}
  
  The proof of this theorem is based on the 
  interpolation operator $\I:C([0,T];H)\to\U$ in time, which uses locally on $I_m$
  the Gauss-Radau points $\tmi$ of the quadrature rule $\Qm{\cdot}$:
  \begin{gather}\label{eq:inter:I}
    (\I_m v-v)(\tmi)=0,\qmbox{for all}m\in\{1,\dots,M\},\,i\in\{0,\dots,q\}.
  \end{gather}
  As a consequence it holds for all $v\in C([0,T];H)$
  \begin{equation}\label{eq:QI}
    \Qm{\I_m v}=\Qm{v}.
  \end{equation}

  We will now present two small extensions of the given convergence result.
  First, in the above result the jumps of the error are not estimated. This will be done next.
  \begin{thm}
    Under the assumptions of Theorem \ref{thm:cts} it holds
    \[
%       \sum_{m=1}^M\e^{-2\rho t_{m-1}}\scp{M_0\jump{U-U^\tau}_{m-1}^0,\jump{U-U^\tau}_{m-1}^0}
      \sum_{m=1}^M\e^{-2\rho t_{m-1}}\norm{M_0^{1/2}\jump{U-U^\tau_h}_{m-1}^0}{H}^2
        \leq CT\tau^{2q+1}.
    \]
  \end{thm}
  \begin{proof}
    Let us decompose the error as
    \[
      U-U^\tau_h=(U-\I U)-(U^\tau_h-\I U)=:\eta+\xi.
    \]
    The proof of \cite[Theorem 4.7]{FrTW16} already gives
    \begin{align*}
%       \sum_{m=1}^M\e^{-2\rho t_{m-1}}\scp{M_0\jump{\xi}_{m-1}^0,\jump{\xi}_{m-1}^0}
      \sum_{m=1}^M\e^{-2\rho t_{m-1}}\norm{M_0^{1/2}\jump{\xi}_{m-1}^0}{H}^2
        &\leq C_1\tau^{2(q+1)}.%\sup_{t\in [0,T]}\norm{\partial_t^{q+2}U(t)}H^2.
    \end{align*}
    The jumps of the interpolation error $\eta$ can be rewritten as
    \begin{align*}
%       \sum_{m=1}^M\!\e^{-2\rho t_{m-1}}\!\scp{M_0\jump{\eta}_{m-1}^0,\jump{\eta}_{m-1}^0}
%       =\sum_{m=1}^M\!\e^{-2\rho t_{m-1}}\!\scp{M_0\jump{U-\I U}_{m-1}^0,\jump{U-\I U}_{m-1}^0}_H\\
%       =\sum_{m=1}^M\e^{-2\rho t_{m-1}}\scp{M_0(U-\I_m U)(t_{m-1}^+),(U-\I_m U)(t_{m-1}^+)}
      \sum_{m=1}^M\e^{-2\rho t_{m-1}}&\norm{M_0^{1/2}\jump{\eta}_{m-1}^0}{H}^2\\
        &\leq\sum_{m=1}^M\!\e^{-2\rho t_{m-1}}(\norm{M_0^{1/2}\jump{U-IU}_{m-1}^0}{H}+\norm{M_0^{1/2}\jump{IU-\I IU}_{m-1}^0}{H})^2\\
        &=\sum_{m=1}^M\e^{-2\rho t_{m-1}}\norm{M_0^{1/2}(IU-\I_m IU)(t_{m-1}^+)}{H}^2
    \end{align*}
    due to $U-IU$ being continuous and the interpolation property $\I_{m-1} IU(t_{m-1}^-)=IU(t_{m-1}^-)$ in each time node $t_{m-1}$.
    Using \cite[Lemma 3.14]{FrTW16} to estimate the interpolation error at the discrete times
    and 
    $
      \sum_{m=1}^M\tau_m\e^{-2\rho t_{m-1}}\leq T,
    $
%     \begin{align*}
%       \scp{M_0(U-\I U)(t_{m-1}^+),(U-\I U)(t_{m-1}^+)}
%         &\leq C\tau_m^{2(q+1)}\sup_{t\in I_m}\norm{\pt_t^{q+1}U(t)}H^2
%     \end{align*}
%     and by
%     \begin{align*}
%       \sum_{m=1}^M\tau_m\e^{-2\rho t_{m-1}}
%       \leq \sum_{m=1}^M\tau_m
%       = T
%     \end{align*}
    it follows
    \begin{align*}
%       \sum_{m=1}^M\e^{-2\rho t_{m-1}}\scp{M_0\jump{\eta}_{m-1}^0,\jump{\eta}_{m-1}^0}
      \sum_{m=1}^M\e^{-2\rho t_{m-1}}\norm{M_0^{1/2}\jump{\eta}_{m-1}^0}{H}^2
        \leq CT\tau^{2q+1}.%\sup_{t\in (0,T)}\norm{\pt_t^{q+1}U(t)}H^2.
    \end{align*}
    Combining the estimates finishes the proof.
  \end{proof}
  \begin{rem}\label{rem:energy:dG}
    The above theorem shows that the jumps converge with half an order less than the function itself.
    As a consequence the dissipation of energy, mentioned in \cite[Rem. 3.3]{FrTW16},
    seen by comparing
    \[
%      \e^{-2\rho t}\scp{M_0U(t),U(t)}
     \e^{-2\rho t}\norm{M_0^{1/2}U(t)}{H}^2
       +2\langle(\rho M_0+M_1)U,U\rangle_{\rho,(0,t)}
%        =\scp{M_0U(0),U(0)}.
       =\norm{M_0^{1/2}U(0)}{H}^2.
    \]
    for the exact solution $U$ at $t$ and
    \begin{multline*}
%      \e^{-2\rho t_i}\scp{M_0{U^\tau_h}(t_i^-),{U^\tau_h}(t_i^-)}
     \e^{-2\rho t_i}\norm{M_0^{1/2}{U^\tau_h}(t_i^-)}{H}^2
       +2\langle(\rho M_0+M_1)U^\tau_h,U^\tau_h\rangle_{\rho,(0,t_i)}\\
%        +2\sum_{m=1}^i \e^{-2\rho t_{m-1}}\scp{M_0 \jump{U^\tau_h}_{m-1}^{x_0},\jump{U^\tau_h}_{m-1}^{x_0}}
       +2\sum_{m=1}^i \e^{-2\rho t_{m-1}}\norm{M_0^{1/2} \jump{U^\tau_h}_{m-1}^{x_0}}{H}^2
%        =\scp{M_0{U^\tau_h}(0^-),{U^\tau_h}(0^-)}
       =\norm{M_0^{1/2} U^\tau_h(0^-)}{H}^2
    \end{multline*}
    for the dG-method's solution $U^\tau_h$ at the discrete times $t_i>0$,
    is converging to zero with order $q+1/2$ and thus the method is asymptotically energy preserving 
    for all $q\geq 0$.
  \end{rem}

  As a second improvement we provide an estimate of the interpolation error in the weighted norm and 
  thereby a convergence result in the full weighted norm.
  \begin{thm}\label{thm:fullcts}
    Under the assumptions of Theorem \ref{thm:cts} it holds 
    \[
      \norm{u-\I_m u}{\rho,m}\leq C \tau_m^{q+1}\norm{u}{H_\rho^{q+1}(I_m;H)},
    \]
    and for the error of the numerical solution
    \[
      \norm{U-U_h^\tau}{\rho}
      \leq C (T^{1/2}\tau^{q+1} + h^{k}).
    \]    
  \end{thm}
  \begin{proof}
    Let us start with the interpolation error estimate.
    We use the transformation $\hat t = (t-t_{m-1})/\tau_m$ from $I_m$ 
    to the reference interval $[0,1]$. Note that the weight function $w(t)=\exp(-2\rho(t-t_{m-1}))$ on $I_m$ is transformed into 
    \[
      \hat w(\hat t)=\e^{-2\rho\tau_m\hat t}.
    \]
%     With Lemma~\ref{lem:app1} we conclude, there exists a $v\in\PS_{\ell-1}(I)$ such that
%     \[
%       \norm{u-v}{H_\rho^{\ell}(I)}\leq c_2|u|_{H_\rho^{\ell}(I)}.
%     \]
%     
    With $\ell=q+1$ in Lemma~\ref{lem:app2} and the transformation we have 
    \[
      \norm{U-\I_m U}{L_\rho^2(I_m)}\leq C\tau_m^{q+1}|U|_{H_\rho^{q+1}(I_m;H)}
    \]
    and upon summation
    \begin{align*}
      \norm{U-\I U}{\rho,[0,T]}^2
        &=\sum_{m=1}^M\e^{2\rho t_{m-1}}\norm{U-\I U}{L_\rho^2(I_m;H)}^2
        \leq C\tau^{2(q+1)}\sum_{m=1}^M\norm{\e^{-2\rho t}\pt_t^{q+1} U}{L^2(I_m;H)}^2\\
        &= C\tau^{2(q+1)}|U|_{H_\rho^{q+1}([0,T];H)}^2,
    \end{align*}
    where the local and global definition of the weights were used.
    
    Now we can estimate the convergence error by
    \[
      \norm{U-U_h^\tau}{\rho}
        \leq \norm{\I U-U_h^\tau}{\rho}+\norm{U-\I U}{\rho}
        = \norm{U-U_h^\tau}{Q,\rho}+\norm{U-\I U}{\rho}
    \]
    due to \eqref{eq:QI}.
    Then the first term is estimated by Theorem~\ref{thm:cts} and the second by the interpolation error estimate above.
  \end{proof}
  The two auxiliary lemmas needed in above proof follow the classical ideas of Bramble and Hilbert~\cite{BH71}, see also \cite{Apel99}. 
  \begin{lem}\label{lem:app1}
    Let $\ell\geq1$ and $u\in H_\rho^\ell([0,1])$. Then it holds
    \[
      |u|_{H_\rho^{\ell}(I)}
        \leq \inf_{v\in\PS_{\ell-1}([0,1])}\norm{u-v}{H_\rho^\ell([0,1])}
        \leq c_2|u|_{H_\rho^\ell([0,1])}.
    \]
  \end{lem}
  \begin{proof}
    The first inequality follows directly from definition of the norms and $v\in\PS_{\ell-1}([0,1])$
    \[
      \inf_{v\in\PS_{\ell-1}([0,1])}\norm{u-v}{H_\rho^\ell([0,1])}
       = \inf_{v\in\PS_{\ell-1}([0,1])}\left(|u|_{H_\rho^{\ell}([0,1])}^2+\norm{u-v}{H_\rho^{\ell-1}([0,1])}^2 \right)^{1/2}
       \geq |u|_{H_\rho^{\ell}([0,1])}.
    \]
    The other inequality is proven by contradiction as in \cite{BH71}.
  \end{proof}
  \begin{lem}\label{lem:app2}
    Let $1\leq \ell\leq q+1$, $\hat \I$ be the mapped Lagrange interpolation operator using the points 
    $t_i\in [0,1]$, $i\in\{0,\dots,q\}$ and $\hat u\in H_\rho^{\ell}([0,1])$. Then it holds
    \[
      \norm{\hat u-\hat \I \hat u}{L_\rho^2([0,1])}
        \leq C |\hat u|_{H_\rho^\ell([0,1])}.
    \]
  \end{lem}
  \begin{proof}
    Let us define $q+1$ functionals $F_i$ using the interpolation points $t_i$ by
    \[
      T_i(\hat v)=\hat v(t_i).
    \]
    Then it holds
    \begin{align*}
      \norm{\hat v-\hat \I \hat u}{L_\rho^2([0,1])}
         =\norm{\hat w(\hat v-\hat \I \hat u)}{L^2([0,1])}
        &\leq c_1\sum_{i=0}^q |F_i(\hat w(\hat v-\hat \I \hat u))|\\
        &= c_1 \sum_{i=0}^q |F_i(\hat w(\hat v-\hat u))|\\
        &\leq c_1c_3\norm{\hat w(\hat v-\hat u)}{H^\ell([0,1])}
         \leq c_1c_3c_4\norm{\hat v-\hat u}{H_\rho^\ell([0,1])}
    \end{align*}
    where we have used the equivalence of norms in finite dimensional spaces in the first inequality,
    then the consistency of $\hat \I \hat u(t_i)=\hat u(t_i)$ and that all $T_i$ are continuous linear functionals in $H^{\ell}([0,1])$.
    The last inequality uses the product rule of differentiation with an exponential factor
    \[
      (\hat w\hat v)'=\hat w'\hat v+\hat w\hat v'=\hat w(-2\rho\tau \hat v+\hat v'),
    \]
    and similarly for higher derivatives. Together with Lemma~\ref{lem:app1} we have
    \[
      \norm{\hat u-\hat \I \hat u}{L_\rho^2([0,1])}
        \leq \norm{\hat u-\hat v}{L_\rho^2([0,1])}+\norm{\hat v-\hat \I \hat u}{L_\rho^2([0,1])}
        \leq c_2(1+c_1c_3c_4)|\hat u|_{H_\rho^\ell([0,1])}.\qedhere
    \]
  \end{proof}

  \section{Dependence of the numerical solution on $\rho$}\label{sec:deprho}
  From \cite[Theorem 1.3.6]{McGPTW20} we know that the exact solution $U$ is independent of $\rho\geq\rho_0$. 
  In this section we want to look at the dependence on $\rho$ for the discrete solutions.
  
  \begin{thm}
    Let $\rho_2>\rho_1\geq\rho_0$ and denote by $U_{h,\rho_k}^\tau\in \U$ the discrete solutions to:
    Find $U_{h,\rho_k}\in\U$ such that for all $m\in\{1,\dots,M\}$ and $\Phi\in\U$ it holds
    \[
      \Qmr[\rho_k]{(\partial_t M_0+M_1+A)U^\tau_{h,\rho_k},\Phi}
        +\scp{M_0 \jump{U^\tau_{h,\rho_k}}_{m-1}^{x_0},{\Phi}^+_{m-1}}
        =\Qmr[\rho_k]{ F,\Phi }.
    \]
    Then it holds
    \[
      \norm{U_{\rho_1}-U_{\rho_2}}{\rho_2}
        \leq C (T^{1/2}\tau^{q+1}+h^{k}).
    \]
  \end{thm}
  \begin{proof}
    This is a direct consequence of Theorem~\ref{thm:cts}:
    \[
      \norm{U_{\rho_1}-U_{\rho_2}}{\rho_2}
        \leq \norm{U_{\rho_1}-U}{\rho_1}+\norm{U_{\rho_2}-U}{\rho_2}
        \leq C (T^{1/2}\tau^{q+1}+h^{k}),
    \]
    due to 
    \[
      \norm{V}{\rho_2}\leq \max\{\e^{(\rho_1-\rho_2)T},1\}\norm{V}{\rho_1}
    \]
    on finite time intervals and $\rho_2>\rho_1$.
  \end{proof}

  Thus, even if the approximations depend on the choice of $\rho$, the difference converges as fast
  as the approximations themselves.  
  Let us take a look at two examples, both with known exact solutions, and compute their numerical solutions. 
  The first is taken from \cite{FrTW16}.
  
  \textbf{Example 1}: 
  Let $\Omega=\left(-\frac{3\pi}{2},\frac{3\pi}{2}\right)\subset\R$,
  $\Omega_{\mathrm{hyp}}= \left(-\frac{3\pi}{2},0\right)$, 
  $\Omega_{\mathrm{par}}= \left(0,\frac{3\pi}{2}\right)$. 
  The problem is given on $\R\times\Omega$ by
%   \begin{subequations}\label{eq:prob1}
    \[%begin{gather}
      \left( 
      \partial_t
      \begin{pmatrix}[c]
        1 & 0\\
        0 & \id{\Omega_\mathrm{hyp}}
      \end{pmatrix}
      +\begin{pmatrix}[c]
        0 & 0\\
        0 & \id{\Omega_\mathrm{par}}
      \end{pmatrix}
      +\begin{pmatrix}[c]
        0 & \partial_x\\
        \partial_x^\circ & 0
      \end{pmatrix}
      \right)\begin{pmatrix}
        u\\
        v
      \end{pmatrix}
      =\begin{pmatrix}
        f\\
        g
      \end{pmatrix}
    \]%end{gather}
    with $u(t,-\frac{3\pi}{2})=u(t,\frac{3\pi}{2})=0$ and right hand-sides $f,\,g$, such that
    the exact solutions are
    \begin{align*}
      u(t,x) &= \id{\R_{\geq0}}(t)(\e^t-1)
                  \left( \id{\left(\frac{\pi}{2},\frac{3\pi}{2}\right)}(x)
                  -\id{\left(-\frac{3\pi}{2},\frac{\pi}{2}\right)}(x)\right)\cos(x),\\
      v(t,x) &= \id{\R_{\geq0}}(t)
                  \left(
                    \!-(\e^t-t-1)\id{\left( -\frac{3\pi}{2},0\right) }(x)\sin(x)
                    \!+\!\id{(0,\pi)}(x)x
                    +\id{\left(\pi,\frac{3\pi}{2}\right)}(x)(2\pi-x)
                  \right)\!.
    \end{align*}    
%   \end{subequations}
  Note that $u$ and $v$ are non-differentiable, but piecewise smooth. In turn
  the right-hand sides $f$ and $g$ are only in $L^2$ in space.
  Figure~\ref{fig:prob1}
  \begin{figure}[tb]
    \begin{center}
\begin{tikzpicture}

\begin{axis}[%
width=0.6\textwidth,
height=0.15\textwidth,
scale only axis,
xmin=-5,
xmax=5,
ymin=-5,
ymax=5,
% xlabel = $x$,
% separate axis lines,
% every outer x axis line/.append style={black},
% every x tick label/.append style={font=\color{black}},
% every outer y axis line/.append style={black},
% every y tick label/.append style={font=\color{black}},
xtick={-4.71,-3.14,-1.57,0,1.57,3.14,4.71},
xticklabels={},%$-\frac{3\pi}{2}$,$-\pi$,$-\frac{\pi}{2}$,0,$\frac{\pi}{2}$,$\pi$,$\frac{3\pi}{2}$},
axis background/.style={fill=white},
legend style={at={(0.97,0.03)},anchor=south east,legend cell align=left,align=left,draw=black}
]
\addplot [color=red,thick]
  table[row sep=crcr]{%
-4.7124	0\\
-4.5553	0.2688\\
-4.3982	0.53098\\
-4.2412	0.78008\\
-4.0841	1.01\\
-3.927	1.215\\
-3.7699	1.3901\\
-3.6128	1.531\\
-3.4558	1.6342\\
-3.2987	1.6971\\
-3.1416	1.7183\\
-2.9845	1.6971\\
-2.8274	1.6342\\
-2.6704	1.531\\
-2.5133	1.3901\\
-2.3562	1.215\\
-2.1991	1.01\\
-2.042	0.78008\\
-1.885	0.53098\\
-1.7279	0.2688\\
-1.5708	0\\
-1.4137	-0.2688\\
-1.2566	-0.53098\\
-1.0996	-0.78008\\
-0.94248	-1.01\\
-0.7854	-1.215\\
-0.62832	-1.3901\\
-0.47124	-1.531\\
-0.31416	-1.6342\\
-0.15708	-1.6971\\
0	-1.7183\\
0.15708	-1.6971\\
0.31416	-1.6342\\
0.47124	-1.531\\
0.62832	-1.3901\\
0.7854	-1.215\\
0.94248	-1.01\\
1.0996	-0.78008\\
1.2566	-0.53098\\
1.4137	-0.2688\\
1.5708	0\\
1.7279	-0.2688\\
1.885	-0.53098\\
2.042	-0.78008\\
2.1991	-1.01\\
2.3562	-1.215\\
2.5133	-1.3901\\
2.6704	-1.531\\
2.8274	-1.6342\\
2.9845	-1.6971\\
3.1416	-1.7183\\
3.2987	-1.6971\\
3.4558	-1.6342\\
3.6128	-1.531\\
3.7699	-1.3901\\
3.927	-1.215\\
4.0841	-1.01\\
4.2412	-0.78008\\
4.3982	-0.53098\\
4.5553	-0.2688\\
4.7124	0\\
};
\addlegendentry{$u$};

\addplot [color=blue,thick, dashed]
  table[row sep=crcr]{%
-4.7124	-4.3891\\
-4.5553	-4.335\\
-4.3982	-4.1742\\
-4.2412	-3.9107\\
-4.0841	-3.5508\\
-3.927	-3.1035\\
-3.7699	-2.5798\\
-3.6128	-1.9926\\
-3.4558	-1.3563\\
-3.2987	-0.6866\\
-3.1416	0\\
-2.9845	0.6866\\
-2.8274	1.3563\\
-2.6704	1.9926\\
-2.5133	2.5798\\
-2.3562	3.1035\\
-2.1991	3.5508\\
-2.042	3.9107\\
-1.885	4.1742\\
-1.7279	4.335\\
-1.5708	4.3891\\
-1.4137	4.335\\
-1.2566	4.1742\\
-1.0996	3.9107\\
-0.94248	3.5508\\
-0.7854	3.1035\\
-0.62832	2.5798\\
-0.47124	1.9926\\
-0.31416	1.3563\\
-0.15708	0.6866\\
0	0\\
3.1416	3.1416\\
4.7124	1.5708\\
};
\addlegendentry{$v$};
\end{axis}
\end{tikzpicture}%
\\
\begin{tikzpicture}

\begin{axis}[%
width=0.6\textwidth,
height=0.15\textwidth,
scale only axis,
xmin=-5,
xmax=5,
ymin=-5,
ymax=4,
% separate axis lines,
% every outer x axis line/.append style={black},
% every x tick label/.append style={font=\color{black}},
% every outer y axis line/.append style={black},
% every y tick label/.append style={font=\color{black}},
xtick={-4.71,-3.14,-1.57,0,1.57,3.14,4.71},
xticklabels={$-\frac{3\pi}{2}$,$-\pi$,$-\frac{\pi}{2}$,0,$\frac{\pi}{2}$,$\pi$,$\frac{3\pi}{2}$},
axis background/.style={fill=white},
legend style={at={(0.03,0.03)},anchor=south west,legend cell align=left,align=left,draw=black}
]
\addplot [color=red, thick]
  table[row sep=crcr]{%
-4.7124	6.3129e-16\\
-4.5553	0.5376\\
-4.3982	1.062\\
-4.2412	1.5602\\
-4.0841	2.02\\
-3.927	2.43\\
-3.7699	2.7802\\
-3.6128	3.062\\
-3.4558	3.2684\\
-3.2987	3.3943\\
-3.1416	3.4366\\
-2.9845	3.3943\\
-2.8274	3.2684\\
-2.6704	3.062\\
-2.5133	2.7802\\
-2.3562	2.43\\
-2.1991	2.02\\
-2.042	1.5602\\
-1.885	1.062\\
-1.7279	0.5376\\
-1.5708	-2.1043e-16\\
-1.4137	-0.5376\\
-1.2566	-1.062\\
-1.0996	-1.5602\\
-0.94248	-2.02\\
-0.7854	-2.43\\
-0.62832	-2.7802\\
-0.47124	-3.062\\
-0.31416	-3.2684\\
-0.15708	-3.3943\\
0 -3.4366\\
};
\addlegendentry{$f$};
\addplot [color=red, thick, forget plot]
  table[row sep=crcr]{%
0	-1.7183\\
0.15708	-1.6848\\
0.31416	-1.5852\\
0.47124	-1.422\\
0.62832	-1.1991\\
0.7854	-0.92212\\
0.94248	-0.59777\\
1.0996	-0.23407\\
1.2566	0.16\\
1.4137	0.57477\\
1.5708	1\\
1.7279	0.57477\\
1.885	0.16\\
2.042	-0.23407\\
2.1991	-0.59777\\
2.3562	-0.92212\\
2.5133	-1.1991\\
2.6704	-1.422\\
2.8274	-1.5852\\
2.9845	-1.6848\\
3.1416	-1.7183\\
};
\addplot [color=red, thick, forget plot]
  table[row sep=crcr]{%
3.1416	-3.7183\\
3.2987	-3.6848\\
3.4558	-3.5852\\
3.6128	-3.422\\
3.7699	-3.1991\\
3.927	-2.9221\\
4.0841	-2.5978\\
4.2412	-2.2341\\
4.3982	-1.84\\
4.5553	-1.4252\\
4.7124	-1\\
};

\addplot [color=blue, thick, dashed]
  table[row sep=crcr]{%
-4.7124	0\\
0	0\\
0.15708	0.42588\\
0.31416	0.84514\\
0.47124	1.2513\\
0.62832	1.6383\\
0.7854	2.0004\\
0.94248	2.3326\\
1.0996	2.6306\\
1.2566	2.8908\\
1.4137	3.1108\\
1.5708	3.2890\\
};
\addlegendentry{$g$};
\addplot [color=blue, thick, dashed, forget plot]
  table[row sep=crcr]{%
1.5708	-0.14749\\
1.7279	0.030749\\
1.885	0.25077\\
2.042	0.51103\\
2.1991	0.809\\
2.3562	1.1412\\
2.5133	1.5033\\
2.6704	1.8903\\
2.8274	2.2965\\
2.9845	2.7157\\
3.1416	3.1416\\
3.2987	3.2533\\
3.4558	3.3584\\
3.6128	3.4504\\
3.7699	3.5233\\
3.927	3.5712\\
4.0841	3.5892\\
4.2412	3.573\\
4.3982	3.5191\\
4.5553	3.425\\
4.7124	3.2891\\
};
\end{axis}
\end{tikzpicture}%
    \end{center}
    \caption{Solutions $u,\,v$ (top) and right-hand sides $f,\,g$ (bottom)
             of Example 1 at $t=1$\label{fig:prob1}}
  \end{figure}
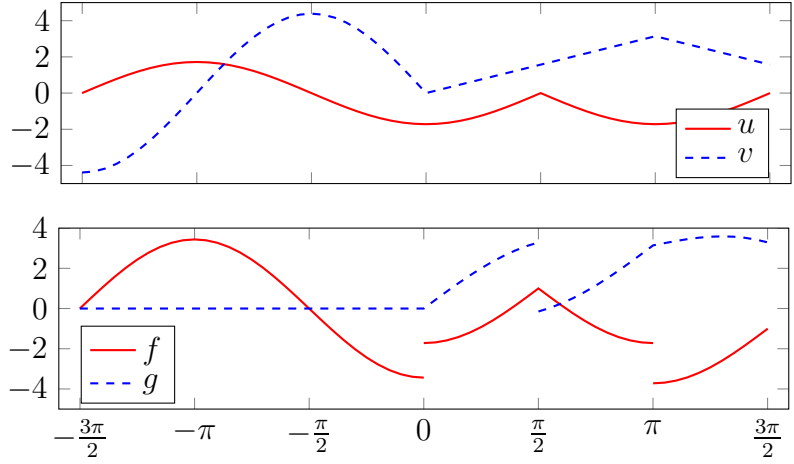
  shows the solutions and the right-hand sides for $t=1$. The value of $\rho_0$ can be any positive number.
  
  \textbf{Example 2}: 
  Let $\Omega=(-1,1)^2\subset\R^2$, $\Omega_{\mathrm{hyp}}=(-1,0)\times(-1,1)$, $\Omega_{\mathrm{ell}}=(0,1)\times(-1,1)$.
  The problem is given on $\R\times\Omega$ by
  \begin{gather}\label{eq:prob3_pde}
    \left( 
    \partial_t
    \begin{pmatrix}[c]
      \id{\Omega_{\mathrm{hyp}}} & 0\\
      0 & \id{\Omega_{\mathrm{hyp}}}
    \end{pmatrix}
    +\begin{pmatrix}[c]
      \id{\Omega_{\mathrm{ell}}} & 0\\
      0 & \id{\Omega_{\mathrm{ell}}}
    \end{pmatrix}
    +\begin{pmatrix}[c]
      0 & \Div \\
      \Grad^\circ & 0
    \end{pmatrix}
    \right)\begin{pmatrix}
      u\\
      v
    \end{pmatrix}
    =\begin{pmatrix}
      f\\
      g
    \end{pmatrix}
  \end{gather}
  with $u=0$ on $\pt\Omega$,
  where $f,g$ are given, such that the exact solution is
  \begin{align*}
    u(t,x,y) &=\id{\R_{\geq0}}(t)(\e^t-1)
                 \left(\id{x<0}\cos(\pi x/2)+\id{x>0}(1-x)\right)(1-|y|)\\
    v(t,x,y) &=\id{\R_{\geq0}}(t)(\e^t-t-1)
                 \big( \id{x<0,y<0}v_{00}(x,y)+\id{x<0,y>0}v_{01}(x,y)\\&\hspace{4cm}
                      +\id{x>0,y<0}v_{10}(x,y)+\id{x>0,y>0}v_{11}(x,y)\big)\\
    v_{00}(x,y) &= \pmtrx{[c]\cos(\pi x/2)y\\\sin(\pi y/2)},\quad 
    v_{01}(x,y)  = \pmtrx{[c]x+y\\-y^2},\\
    v_{10}(x,y) &= \pmtrx{[c]x^2+y\\0},\quad 
    v_{11}(x,y)  = \pmtrx{[c]y(1+x)\\0}.
  \end{align*}

  Figure~\ref{fig:prob2}
%   
%   \tikzset{external/force remake} % remake all
%   \tikzset{external/remake next} % remake next
  \begin{figure}[tb]
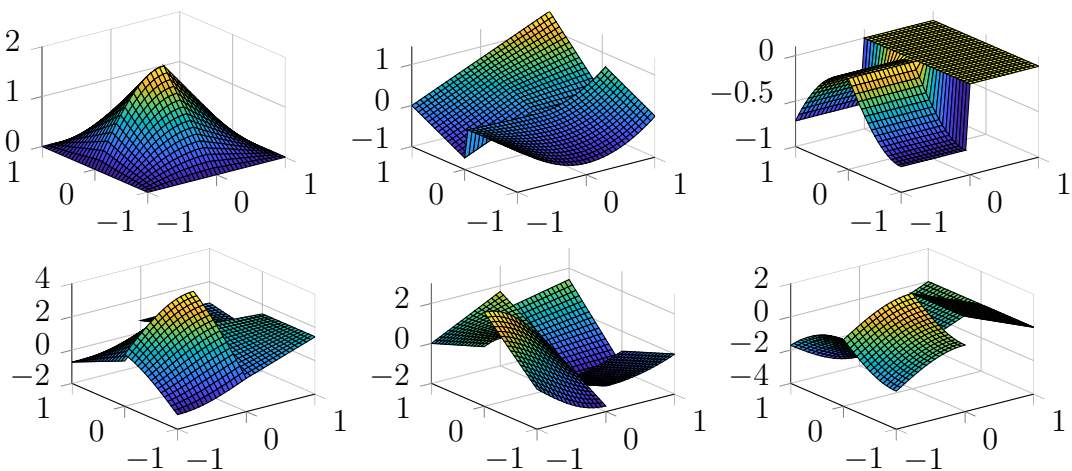

    \begin{center}
       \input{u_prob2}%
       \input{v1_prob2}%
       \input{v2_prob2}\\
       \input{f_prob2}%
       \input{g1_prob2}%
       \input{g2_prob2}
    \end{center}
    \caption{Solutions $u$ and the two components of $v$ (top), and right-hand sides $f$ and the two components of $g$ (bottom)
             of Example 2 at $t=1$\label{fig:prob2}}
  \end{figure}
  shows the solutions and the right-hand sides for $t=1$. Note that the two components of $v$ are not independent of each other
  and together form a function in $H(\Div,\Omega)$. The right-hand sides are only in $L^2$ in space and again we have $\rho_0>0$.

  The computations are done in Matlab using the finite-element framework $\mathbb{SOFE}$, 
  see \texttt{github.com/SOFE-Developers/SOFE}. 
  For our simulations we choose $T=1$, $\rho_1=1$ and $\rho_2=2$ --- both larger than $\rho_0$. 
  
  We set $M=N$ and choose equidistant meshes in space and time. Then $h\sim\tau$ follows.
  We also vary the polynomial degrees and in order to balance the two error contributions,
  we set the polynomial degree in time to be $q=k-1$ for $k\geq 1$. 
  Under these circumstances we can expect the convergence to be of order $q+1=k$ in space and time.
  
  Tables~\ref{tab:prob1:diffrho}
  \begin{table}[tbp]
    \caption{Convergence in weighted $L^2$ of $U_{\rho_1}$ and $U_{\rho_2}$ to $U$ and to each other for Example 1 using various polynomial degrees\label{tab:prob1:diffrho}}
    \begin{center}
      \begin{tabular}{cc|cccccc}
       $k/q$ & $M=N$ & \multicolumn{2}{c}{$\norm{U-U_{\rho_1}}{\rho_1}$} & \multicolumn{2}{c}{$\norm{U-U_{\rho_2}}{\rho_2}$} & \multicolumn{2}{c}{$\norm{U_{\rho_1}-U_{\rho_2}}{\rho_2}$}\\
        \hline
        \multirow{3}{*}{1/0} 
%         &\multicolumn{6}{c}{$k=1, q=0$}\\
%          24 & 4.430e-02 &      & 2.674e-02 &      & 1.968e-02 &     \\
%          48 & 2.194e-02 & 1.01 & 1.336e-02 & 1.00 & 9.872e-03 & 1.00\\
%          96 & 1.094e-02 & 1.00 & 6.688e-03 & 1.00 & 4.945e-03 & 1.00\\
%         192 & 5.462e-03 & 1.00 & 3.347e-03 & 1.00 & 2.475e-03 & 1.00\\
        &192 & 5.462e-03 &      & 3.347e-03 &      & 2.475e-03 &     \\
        &384 & 2.730e-03 & 1.00 & 1.675e-03 & 1.00 & 1.238e-03 & 1.00\\
        &768 & 1.364e-03 & 1.00 & 8.376e-04 & 1.00 & 6.192e-04 & 1.00\\
%         &\multicolumn{6}{c}{$k=2, q=1$}\\
        \hline
        \multirow{3}{*}{2/1} 
%          24 & 1.064e-03 &      & 5.567e-04 &      & 5.842e-06 &     \\
%          48 & 2.657e-04 & 2.00 & 1.392e-04 & 2.00 & 7.295e-07 & 3.00\\
%          96 & 6.641e-05 & 2.00 & 3.481e-05 & 2.00 & 9.114e-08 & 3.00\\
%         192 & 1.660e-05 & 2.00 & 8.703e-06 & 2.00 & 1.139e-08 & 3.00\\
        &192 & 1.660e-05 &      & 8.703e-06 &      & 1.139e-08 &     \\
        &384 & 4.150e-06 & 2.00 & 2.176e-06 & 2.00 & 1.424e-09 & 3.00\\
        &768 & 1.037e-06 & 2.00 & 5.440e-07 & 2.00 & 1.780e-10 & 3.00\\
%         &\multicolumn{6}{c}{$k=3, q=2$}\\ 
        \hline
        \multirow{4}{*}{3/2} 
%          24 & 4.486e-06 &      & 2.323e-06 &      & 1.960e-08 &     \\
%          48 & 2.948e-07 & 3.93 & 1.581e-07 & 3.88 & 1.225e-09 & 4.00\\
%          96 & 2.233e-08 & 3.72 & 1.291e-08 & 3.61 & 7.652e-11 & 4.00\\
       &  96 & 2.233e-08 &      & 1.291e-08 &      & 7.652e-11 &     \\
       & 192 & 2.113e-09 & 3.40 & 1.319e-09 & 3.29 & 4.783e-12 & 4.00\\
%        & 192 & 2.113e-09 &      & 1.319e-09 &      & 4.783e-12 &     \\
       & 384 & 2.384e-10 & 3.15 & 1.543e-10 & 3.10 & 3.002e-13 & 3.99\\
       & 768 & 2.893e-11 & 3.04 & 1.894e-11 & 3.03 & 8.106e-14 & 1.89
      \end{tabular}
    \end{center}
  \end{table}
  and~\ref{tab:prob2:diffrho}
  \begin{table}[tbp]
    \caption{Convergence in weighted $L^2$ of $U_{\rho_1}$ and $U_{\rho_2}$ to $U$ and to each other for Example 2 using various polynomial degrees\label{tab:prob2:diffrho}}
    \begin{center}
      \begin{tabular}{cc|cccccc}
        $k/q$ & $M=N$ & \multicolumn{2}{c}{$\norm{U-U_{\rho_1}}{\rho_1}$} & \multicolumn{2}{c}{$\norm{U-U_{\rho_2}}{\rho_2}$} & \multicolumn{2}{c}{$\norm{U_{\rho_1}-U_{\rho_2}}{\rho_2}$}\\
        \hline
        \multirow{3}{*}{1/0} 
%         &\multicolumn{6}{c}{$k=1, q=0$}\\
%          8 & 7.369e-02 &      & 3.953e-02 &      & 1.594e-02 &     \\
%         16 & 3.614e-02 & 1.03 & 1.935e-02 & 1.03 & 7.863e-03 & 1.02\\
        &16 & 3.614e-02 &      & 1.935e-02 &      & 7.863e-03 &     \\
        &32 & 1.789e-02 & 1.01 & 9.583e-03 & 1.01 & 3.926e-03 & 1.00\\
        &64 & 8.902e-03 & 1.01 & 4.771e-03 & 1.01 & 1.965e-03 & 1.00\\
        \hline
        \multirow{3}{*}{2/1} 
%         &\multicolumn{6}{c}{$k=2, q=1$}\\
%          8 & 1.692e-03 &      & 1.050e-03 &      & 8.365e-05 &     \\
%         16 & 4.213e-04 & 2.01 & 2.631e-04 & 2.00 & 1.042e-05 & 3.00\\
%         32 & 1.051e-04 & 2.00 & 6.583e-05 & 2.00 & 1.299e-06 & 3.00\\
       & 32 & 1.051e-04 &      & 6.583e-05 &      & 1.299e-06 &     \\
       & 64 & 2.626e-05 & 2.00 & 1.646e-05 & 2.00 & 1.622e-07 & 3.00\\
       &128 & 6.563e-06 & 2.00 & 4.117e-06 & 2.00 & 2.026e-08 & 3.00\\
        \hline
        \multirow{3}{*}{3/2} 
%         &\multicolumn{6}{c}{$k=3, q=2$}\\
%          8 & 1.602e-05 &      & 1.021e-05 &      & 8.417e-07 &     \\
%         16 & 1.932e-06 & 3.05 & 1.258e-06 & 3.02 & 5.247e-08 & 4.00\\
        &16 & 1.932e-06 &      & 1.258e-06 &      & 5.247e-08 &     \\
        &32 & 2.393e-07 & 3.01 & 1.569e-07 & 3.00 & 3.273e-09 & 4.00\\
        &64 & 2.985e-08 & 3.00 & 1.960e-08 & 3.00 & 2.043e-10 & 4.00
      \end{tabular}
    \end{center}
  \end{table}
  show the proposed orders of convergence $q+1$ in all cases. For $q\geq 1$ there is an interesting phenomenon,
  where the difference $U_{\rho_1}-U_{\rho_2}$ converges with one order higher than theoretically predicted.
  For that we do not have a proof.
  
%   \TODO{Theory for higher order convergence}
  
  \section{A reformulation of the problem and its numerical approximation}\label{sec:reform}
  One problem in the implementation of the method \eqref{eq:dG} is the choice of the weighted
  Gauß-Radau quadrature rule. While this is a feasible task, see \cite{PTVF07,Waurick16}, 
  the weights and quadrature points depend on $\rho$ and $\tau_m$.
  Thus, on a non-equidistant mesh they have to be computed separately for each interval $I_m$.
  But we can actually reformulate the given problem in such a way that no weighted space is needed for its solution existence 
  and thus no weighted quadrature rule. 
  For that let 
  \[
    V(t):=\e^{-\rho t}U(t)
    \lrarrow
    U(t)=\e^{\rho t}V(t)
  \]
  be the transformed solution. Now $V$ is the solution to the problem
  \[
    (\pt_t M_0+(\rho M_0+M_1)+A)V(t)=\e^{-\rho t}F(t)=:\tilde F(t).
  \]
  This problem has the same structure as the original problem with a modified linear, 
  self-adjoint operator $\rho M_0+M_1$, which is positive by design for $\rho\geq \rho_0$.
  But this in turn means, that we can use the standard $L^2$-space
  \[
    H(\R;H):= 
      \left\{ 
          f:\R\to H\,:\, f \mbox{ meas.}, 
          \int_\R |f(t)|_H^2 \dt<\infty
      \right\}
  \]
  as solution space. Then by Theorem \ref{thm:exist} the uniqueness of a solution $V\in H(\R;H)$ for all $\tilde F\in H(\R;H)$ follows.
  Note that because of the connection between $U$ and $V$ and especially the definition of $\tilde F$, $V$ depends on the choice of $\rho\geq\rho_0$.
  
  We can use a similar numerical method as before by choosing a standard right-sided Gauß-Radau quadrature
  rule $\Qm{\cdot}:=\Qmr[0]{\cdot}$ and the associated discretised scalar product. So the discrete method reads: Find $V_h^\tau\in\U$
  such that for all $\Phi\in\U$ and $m\in\{1,\dots,M\}$ it holds
  \[
    \Qm{(\pt_t M_0+(\rho M_0+M_1)+A)V_h^\tau,\Phi}+\scp{M_0\jump{V_h^\tau}_{m-1}^{x_0},\Phi_{m-1}^+}=\Qm{\tilde F,\Phi}.
  \]
  Similar to the results of Section~\ref{sec:basics} we have the following convergence result.
  \begin{thm}
    We assume for the solution $V(t)=\e^{-\rho t}U(t)$ the 
    regularity     
    \[
      V\in H^{1}(\R;H^k(\Omega))\cap 
      H^{q+3}(\R;L^2(\Omega)) 
    \]
    as well as 
    \[
      AV\in H(\R; H^k(\Omega)).
    \]
    Then we have for the error of the numerical solution
    \[
%       \sup_{t\in[0,T]}\scp{M_0(V-V_h^\tau)(t),(V-V_h^\tau)(t)}+
      \sup_{t\in[0,T]}\norm{M_0^{1/2}(V-V_h^\tau)(t)}{H}+
      \norm{V-V_h^\tau}{0}
      \leq C (T^{1/2}\tau^{q+1} + h^{k}).
    \]
  \end{thm}
  \begin{rem}
    An approximation to the original problem is then given by
    \[
      E_\rho V_h^\tau(t):=\e^{\rho t}V_h^\tau(t)
    \]
    with the estimation
    \[
      \norm{E_\rho V_h^\tau - U}{0}
        = \norm{\e^{\rho t}(V_h^\tau - V)}{0}
        \leq C \e^{\rho T}(T^{1/2}\tau^{q+1}+h^{k}),
    \]
    comparable to the result of Theorem~\ref{thm:cts}. 
    Note that $E_\rho V_h^\tau$ is no longer piecewise polynomial in time, but
    it holds in comparison to the piecewise polynomial solution $U_h^\tau$ from the previous section.
    \begin{align*}
      \norm{E_\rho V_h^\tau-U_h^\tau}{0}
%         &\leq \norm{E_\rho V_h^\tau-U}{0}+\norm{U_h^\tau-U}{0}\\
        &\leq \norm{E_\rho V_h^\tau-U}{0}+\e^{\rho T}\norm{U_h^\tau-U}{\rho}
         \leq C\e^{\rho T}(T^{1/2}\tau^{q+1}+h^{k}).
    \end{align*}
  \end{rem}
  \begin{rem}    
    In general $V_{h,\rho}^\tau:=V_h^\tau$ and also $E_\rho V_{h,\rho}^\tau$ will depend on $\rho$. 
    Furthermore, for different values of $\rho$ the discrete solutions $V_{h,\rho_1}^\tau$ and $V_{h,\rho2}^\tau$
    will in general not converge to each other, since they solve two different problems, but
    \begin{align*}
      \norm{E_{\rho_1} V_{h,\rho_1}^\tau-E_{\rho_2} V_{h,\rho_2}^\tau}{0}
        &\leq \norm{E_{\rho_1} V_{h,\rho_1}^\tau-U}{0}+\norm{E_{\rho_2} V_{h,\rho_2}^\tau-U}{0}\\
        &\leq C \e^{\max\{\rho_1,\rho_2\} T}(T^{1/2}\tau^{q+1}+h^{k}).
    \end{align*}    
  \end{rem}
  
  For our numerical simulations we use again the given two examples from before and $\rho=1$ or $\rho=2$. 
  Tables~\ref{tab:prob1:mod}
   \begin{table}[tbp]
    \caption{Convergence in $L^2$ of $V_h^\tau$ to $V$, $E_\rho V_h^\tau$ to $U$ and comparison of $E_\rho V_h^\tau$ to $U_h^\tau$ for Example 1 using various polynomial degrees and $\rho=2$\label{tab:prob1:mod}}
    \begin{center}
      \begin{tabular}{cc|cccccc}
        $k/q$ & $M=N$ & \multicolumn{2}{c}{$\norm{V-V_h^\tau}{0}$} & \multicolumn{2}{c}{$\norm{U-E_\rho V_h^\tau}{0}$} & \multicolumn{2}{c}{$\norm{E_\rho V_h^\tau-U_h^\tau}{0}$}\\
        \hline
        \multirow{3}{*}{1/0} 
%         &\multicolumn{6}{c}{$k=1, q=0$}\\
%        &  24 &  1.044e-01 &      &  2.097e-01 &       & 2.185e-01 &     \\
%        &  48 &  5.253e-02 & 0.99 &  1.064e-01 & 0.98  & 1.108e-01 & 0.98\\
%        &  96 &  2.634e-02 & 1.00 &  5.361e-02 & 0.99  & 5.577e-02 & 0.99\\
%        & 192 &  1.319e-02 & 1.00 &  2.691e-02 & 0.99  & 2.798e-02 & 0.99\\
       & 192 &  1.319e-02 &      &  2.691e-02 &       & 2.798e-02 & \\
       & 384 &  6.601e-03 & 1.00 &  1.348e-02 & 1.00  & 1.402e-02 & 1.00\\
       & 768 &  3.302e-03 & 1.00 &  6.747e-03 & 1.00  & 7.014e-03 & 1.00\\
        \hline
        \multirow{3}{*}{2/1} 
%         &\multicolumn{6}{c}{$k=2, q=1$}\\
%        &  24 &  1.102e-03 &      &  2.890e-03 &       & 2.055e-03 &     \\
%        &  48 &  2.761e-04 & 2.00 &  7.254e-04 & 1.99  & 5.202e-04 & 1.98\\
%        &  96 &  6.910e-05 & 2.00 &  1.818e-04 & 2.00  & 1.309e-04 & 1.99\\
%        & 192 &  1.728e-05 & 2.00 &  4.550e-05 & 2.00  & 3.283e-05 & 2.00\\
       & 192 &  1.728e-05 &      &  4.550e-05 &       & 3.283e-05 & \\
       & 384 &  4.321e-06 & 2.00 &  1.138e-05 & 2.00  & 8.221e-06 & 2.00\\
       & 768 &  1.080e-06 & 2.00 &  2.846e-06 & 2.00  & 2.057e-06 & 2.00\\
        \hline
        \multirow{3}{*}{3/2} 
%         &\multicolumn{6}{c}{$k=3, q=2$}\\
%        &  24 &  7.081e-06 &      &  1.610e-05 &       & 1.279e-05 &     \\
%        &  48 &  8.523e-07 & 3.05 &  1.754e-06 & 3.20  & 1.617e-06 & 2.98\\
%        &  96 &  1.056e-07 & 3.01 &  2.113e-07 & 3.05  & 2.033e-07 & 2.99\\
%        & 192 &  1.317e-08 & 3.00 &  2.620e-08 & 3.01  & 2.549e-08 & 3.00\\
       & 192 &  1.317e-08 &      &  2.620e-08 &       & 2.549e-08 & \\
       & 384 &  1.645e-09 & 3.00 &  3.271e-09 & 3.00  & 3.191e-09 & 3.00\\
       & 768 &  2.056e-10 & 3.00 &  4.089e-10 & 3.00  & 3.991e-10 & 3.00
      \end{tabular}
    \end{center}
  \end{table}
  and~\ref{tab:prob2:mod}
   \begin{table}[tbp]
    \caption{Convergence in $L^2$ of $V_h^\tau$ to $V$, $E_\rho V_h^\tau$ to $U$ and comparison of $E_\rho V_h^\tau$ to $U_h^\tau$ for Example 2 using various polynomial degrees and $\rho=1$\label{tab:prob2:mod}}
    \begin{center}
      \begin{tabular}{cc|cccccc}
        $k/q$ & $M=N$ & \multicolumn{2}{c}{$\norm{V-V_h^\tau}{0}$} & \multicolumn{2}{c}{$\norm{U-E_\rho V_h^\tau}{0}$} & \multicolumn{2}{c}{$\norm{E_\rho V_h^\tau-U_h^\tau}{0}$}\\
        \hline
        \multirow{3}{*}{1/0}
%        &  8 &  5.242e-02 &       & 9.636e-02 &      &  5.482e-02 &     \\
%        & 16 &  2.637e-02 & 0.99  & 4.900e-02 & 0.98 &  2.641e-02 & 1.05\\
       & 16 &  2.637e-02 &       & 4.900e-02 &      &  2.641e-02 & \\
       & 32 &  1.322e-02 & 1.00  & 2.471e-02 & 0.99 &  1.294e-02 & 1.03\\
       & 64 &  6.621e-03 & 1.00  & 1.241e-02 & 0.99 &  6.399e-03 & 1.02\\
        \hline
        \multirow{3}{*}{2/1}
%        &  8 &  1.041e-03 &      &  1.896e-03 &      &  2.691e-03 &     \\
%        & 16 &  2.617e-04 & 1.99 &  4.758e-04 & 1.99 &  6.818e-04 & 1.98\\
       & 16 &  2.617e-04 &      &  4.758e-04 &      &  6.818e-04 & \\
       & 32 &  6.562e-05 & 2.00 &  1.193e-04 & 2.00 &  1.716e-04 & 1.99\\
       & 64 &  1.643e-05 & 2.00 &  2.988e-05 & 2.00 &  4.304e-05 & 2.00\\
       \hline
        \multirow{3}{*}{3/2}
%       &  8 &  1.613e-05 &      &  2.354e-05 &       & 4.252e-05 &     \\
%       & 16 &  1.962e-06 & 3.04 &  2.752e-06 & 3.10  & 5.406e-06 & 2.98\\
      & 16 &  1.962e-06 &      &  2.752e-06 &       & 5.406e-06 & \\
      & 32 &  2.440e-07 & 3.01 &  3.400e-07 & 3.02  & 6.818e-07 & 2.99\\
      & 64 &  3.048e-08 & 3.00 &  4.249e-08 & 3.00  & 8.560e-08 & 2.99\\
      \end{tabular}
    \end{center}
  \end{table}
  show for the two numerical examples the expected convergence rates for $V_h^\tau$ converging to $V$ and $E_\rho V_h^\tau$ converging to $U$.
  We also compare the approximations $U_h^\tau$ from the previous section and $E_\rho V_h^\tau$. 
  Both and also their difference converge with the same order. So the two approximations are different but equally good.

  \section{Postprocessing and improvement of accuracy}\label{sec:post}
  The numerical solution $V_h^\tau\in\U$ is discontinuous at the discrete times $t_{m-1}$, $m\in\{1,\dots,M\}$.
  For discontinuous Galerkin methods it is known, see \cite{ES16}, that a simple post-processing using the jump at these time-points
  can give an improved approximation. This only makes sense for problems, where the solution is at least continuous in time.
  Under stronger conditions on $F$ we can show that the solution $U$ of the original problem is continuous in time and satisfies
  $U(0^+)=x_0$, see~\cite{Fr19}. To be more precise, we assume
  $F|_{\R_{\geq 0}}$ is continuous, $F(t)=0$, $t<0$, $x_0\in\text{dom}(A)$,
  and
  \[
    (\rho M_0+M_1+A)x_0=F(0^+).
  \]
  
  Now let the globally discontinuous function $\theta$ with $\theta_m\in\PS_{q+1}(I_m)$ be given on any time-interval $I_m$ by
  the conditions $\theta_m(\tmi)=0$, $i\in\{0,\dots,q\}$ and $\theta_m(t_{m-1})=1$. 
%   Then a first post-processed solution $\dtilde V_h^\tau$ is defined locally on $I_m$
%   by
%   \[
%     \dtilde V_h^\tau|_{I_m}:=V_h^\tau-N^{\perp}\jump{V_h^\tau}_{m-1}^{x_0}\theta_m,
%   \]
%   where $N$ is the projector onto the nullspace of $M_0$ and $N^\perp$ its ortho-complement.
%   This function solves for each $m\in\{1,\dots,M\}$ and for all $\Phi\in\U$ the following problem
%   \begin{align*}
%     \Qm{(\pt_t M_0+(\rho M_0+M_1)+A)\dtilde V_h^\tau,\Phi}&=\Qm{\tilde F,\Phi},\\
%     N^{\perp}\jump{\dtilde V_h^\tau}_{m-1}^{x_0}&=0.
%   \end{align*}
%   Note that this problem is only uniquely solvable up to contributions of components of the nullspace 
%   of $M_0$ times $\theta$. 
  
  Then we define the post-processed solution
  \[
    \tilde V_h^\tau|_{I_m}:=V_h^\tau-\jump{V_h^\tau}_{m-1}^{x_0}\theta_m
  \]
  that lies in the space
  \[
    \V:=\left\{
          V\in H^1(\R,H):
          V|_{I_m}\in\PS_{q+1}(I_m,V_k(\Omega)),
          m\in\{1,\dots,M\}
       \right\}.
  \]

  \begin{lem}
    The post-processed solution $\tilde V_h^\tau$ is the unique solution of:
    Find $\tilde V_h^\tau\in\V$ such that for all $m\in\{1,\dots,M\}$ and $\Phi\in\U$ it holds
    \begin{align}\label{eq:dG-C0}
      \Qm{(\pt_t M_0+(\rho M_0+M_1)+A)\tilde V_h^\tau,\Phi}&=\Qm{\tilde F,\Phi},\quad
      \jump{\tilde V_h^\tau}_{m-1}^{x_0}=0.
    \end{align}
  \end{lem}
  \begin{proof}
    This follows directly by substituting $\tilde V_h^\tau$ in above quadrature rule and using, that $V_h^\tau$
    solves the dG-formulation \eqref{eq:dG}. 
  \end{proof}  
  The resulting problem \eqref{eq:dG-C0} is a version of a so-called dG-C0 method~\cite{MS11,BM21,BM22}. 
  \begin{rem}
    The formulation \eqref{eq:dG-C0} immediately gives the conservation of energy by the post-processed solution 
    $\tilde V_h^\tau$ for the transformed problem, compare Remark \ref{rem:energy:dG}.
  \end{rem}
  \begin{rem}
    Choosing $\Phi$ as the Lagrange-basis functions
    corresponding to the Gauß-Radau quadrature points we also have the equivalence of \eqref{eq:dG-C0} to a collocation method, see also \cite[Theorem 1.24]{Becher22}.
    \begin{align*}
      (\pt_t M_0+(\rho M_0+M_1)+A)\tilde V_h^\tau(\tmi)&=\tilde F(\tmi),\quad
      \jump{\tilde V_h^\tau}_{m-1}^{x_0}=0
    \end{align*}
    for all $m\in\{1,\dots,M\}$ and $i\in\{0,\dots,q\}$. 
  \end{rem}
  In \cite{MS11} we find a convergence result for the dG-C0-method in $L^\infty$ in time
  in the case of an ode using a discrete Gronwall lemma. We will take a different approach here and start by defining 
  two more interpolation operators, see \cite{ES16}, besides the Gauß-Radau interpolation operator $\I$ of \eqref{eq:inter:I}.
  
  The first operator $\K:H^2(\R;H)\to \V[q+2]$ is defined for $q\geq 1$ locally on $I_m$ by
  \begin{subequations}\label{eq:inter:K}
  \begin{align}
    (\K_m v-v)(t_{n-1}^+)&=0,  &(\K_m v-v)(t_n^-)&=0,\\
    \pt_t(\K_m v-v)(t_{n-1}^+)&=0, &\pt_t(\K_m v-v)(t_n^-)&=0,
  \end{align}
  \end{subequations}
  plus $q-1$ additional independent Lagrange-conditions in the interior of $I_m$. %, e.g. in some of the interior Gauß-Radau points $\tmi$.
  Assuming $v$ is smooth enough we have by standard interpolation error estimates
  \begin{subequations}\label{eq:intertime}
  \begin{gather}
    \norm{\pt_t(v-\K v)}{L^\infty([0,T])}\leq C \tau^{q+2}\norm{v}{W^{q+3,\infty}(0,T)}.
  \end{gather}
  The second interpolation operator $\J:H^2(\R;H)\to\PS_{q+1}(\R;H)$ is defined locally by
  \[
    (\J_m v-v)(t_{m-1}^+)=0\qmbox{and}
    \pt_t(\J_m v-\K v)(\tmi)=0,\,i\in\{0,\dots,q\}.
  \]
  Note that by \cite[Lemma 5]{ES16} $\J v$ is continuous with $\J_m v(t_m^-)=v(t_m)$ and
  \begin{gather}
    \norm{v-\J v}{L^\infty([0,T])}\leq C \tau^{q+2}\norm{v}{W^{q+2,\infty}(0,T)}.
  \end{gather}
  \end{subequations}
  Also by continuity it holds for $\xi\in\PS_{q+1}^{cont}(H;\R)$ on $I_m$
  \begin{align}
    \xi &= \I_m \xi+(\xi(t_{m-1})-\I_m \xi(t_{m-1}^+))\theta_m\notag\\
        &= \I_m \xi+(\I_{m-1} \xi(t_{m-1}^-)-\I_m\xi(t_{m-1}^+))\theta_m\notag\\
        &= \I_m \xi-\jump{\I \xi}_{m-1}^{\xi(0)}\theta_m.\label{eq:J_jump_I}
  \end{align}

  \begin{thm}\label{thm:post}
    Let $q\geq 1$ and $\tau_m/\tau_{m-1}\leq C$ for all $m\in\{1,\dots,M\}$. Then the error between the post-processed solution $\tilde V^\tau_h$ and $V$
    can be estimated by
    \[
      \norm{\tilde V^\tau_h-V}{0}\leq CT^{1/2}(\tau^{q+2}+h^{k}).
    \]
  \end{thm}
  \begin{proof}
    Since the exact solution $V$ also satisfies \eqref{eq:dG-C0}, we have Galerkin orthogonality. Together with the splitting
    \[
      \tilde V_h^\tau-V=\xi-\eta,\quad
      \xi:=\tilde V_h^\tau-\J I V,\,
      \eta:=\eta_\tau+\eta_h,\,
      \eta_\tau:=V-\J V,\,
      \eta_h:=\J V-I\J V
    \]
    into discrete and interpolation errors in time and space we obtain the error equation
    \begin{equation}\label{eq:erroreq}
      \Qm{(\pt_t M_0+(\rho M_0+M_1)+A)\eta,\Phi}
        = \Qm{(\pt_t M_0+(\rho M_0+M_1)+A)\xi,\Phi}
    \end{equation}
    for all $\Phi\in\U$. If we choose $\Phi=\I\xi$, we can rewrite the right-hand side as
    \begin{multline*}
      \Qm{(\pt_t M_0+(\rho M_0+M_1)+A)\xi,\I_m\xi}\\
         = \Qm{\pt_t M_0\xi,\I\xi}+\Qm{((\rho M_0+M_1)+A)\I_m\xi,\I_m\xi}
    \end{multline*}
    where we have used \eqref{eq:QI}. With $A=-A^*$ and 
    $\rho M_0+M_1\geq\gamma$ we get
    \[
      \Qm{(\rho M_0+M_1)\I_m\xi,\I_m\xi}+\Qm{A\I_m\xi,\I_m\xi}
        \geq \gamma \norm{\I_m\xi}{L^2(I_m)}^2.
    \]
    The term involving the time derivative can be rewritten using the equivalence of the 
    quadrature rule and the scalar product, and integration by parts as
    \[
      \Qm{\pt_t M_0\xi,\I_m\xi}
%         = \frac{1}{2}\left( \norm{M_0\jump{\I\xi}_{m-1}}H^2+\norm{\I_m\xi(t_m^-)}H^2-\norm{\I_{m-1}\xi(t_{m-1}^-)}H^2 \right).
        = \frac{1}{2}\left( \norm{M_0^{1/2}\jump{I\xi}_{m-1}}{H}^2 + \norm{M_0^{1/2}I_m\xi(t_m^-)}{H}^2 - \norm{M_0^{1/2}I_{m-1}\xi(t_{m-1}^-)}{H}^2 \right).
  \]
%     {\color{red}
%   Now for the first scalar product we use integration by parts and obtain
%   \begin{align*}
%     \scp[m]{\pt_t M_0\xi,I_m\xi}
%       &=\scp{M_0 \xi,I_m\xi}|_{t_{m-1}}^{t_m}-\scp[m]{M_0\xi,\pt_t I_m\xi}
%   \intertext{where the remaining scalar product can be written in quadrature form}
%       &=\scp{M_0 \xi,I_m\xi}|_{t_{m-1}}^{t_m}-\Qm{M_0\xi,\pt_t I_m\xi}
%        =\scp{M_0 \xi,I_m\xi}|_{t_{m-1}}^{t_m}-\Qm{M_0I_m\xi,\pt_t I_m\xi}\\
%       &=\scp{M_0 \xi,I_m\xi}|_{t_{m-1}}^{t_m}-\scp[m]{M_0 I_m\xi,\pt_t I_m\xi}
%   \intertext{and another integration by parts yields}
%       &=\scp{M_0 \xi,I_m\xi}|_{t_{m-1}}^{t_m}-\scp{M_0 I_m\xi,I_m\xi}|_{t_{m-1}}^{t_m}+\scp[m]{\pt_tM_0 I_m\xi, I_m\xi}
%   \end{align*}
%   With
%   \[
%     \scp[m]{\pt_tM_0 I_m\xi, I_m\xi}
%       =\frac{1}{2}\int_{I_m}\pt_t\scp{M_0 I_m\xi,I_m\xi}\dt
%       =\frac{1}{2}\left( \norm{M_0^{1/2}I_m\xi(t_m^-)}{H}^2-\norm{M_0^{1/2}I_m\xi(t_{m-1}^+)}{H}^2 \right)
%   \]
%   and 
%   \[
%     \xi(t_{m-1}^+)=\xi(t_{m-1}^-)=I_{m-1}\xi(t_{m-1}^-)
%   \]
%   it follows
%   \[
%     \scp[m]{\pt_tM_0\xi,I_m\xi}
%       = \frac{1}{2}\left( \norm{M_0^{1/2}\jump{I\xi}_{m-1}}{H}^2 + \norm{M_0^{1/2}I_m\xi(t_m^-)}{H}^2 - \norm{M_0^{1/2}I_{m-1}\xi(t_{m-1}^-)}{H}^2 \right)
%   \]}
    Summing over $m$ we get
    \begin{multline*}
      \gamma\norm{\I\xi}{0}^2+\frac{1}{2}\sum_{m=1}^M\norm{M_0^{1/2}\jump{\I\xi}_{m-1}}H^2+\frac{1}{2}\norm{M_0^{1/2}\xi(T)}H^2\\
        \leq \sum_{m=1}^M\Qm{(\pt_t M_0+(\rho M_0+M_1)+A)\eta,\I_m\xi}. 
    \end{multline*}
    We estimate the right hand side as follows
    \begin{align*}
      \Qm{\pt_t M_0\eta_\tau,\I_m\xi}
        &= \Qm{\pt_t M_0(V-\J_m V),\I_m\xi}\\
        &= \Qm{\pt_t M_0(V-\K_m V),\I_m\xi}
         \leq \norm{\pt_t M_0(V-\K_m V)}{Q,m}\norm{\I_m\xi}{Q,m},\\
      \Qm{\pt_t M_0\eta_h,\I_m\xi}
        &= \Qm{\pt_t M_0(\J_m V-I\J_m V),\I_m\xi}\\
        &\leq \norm{\pt_t M_0\J_m V-I\pt_t M_0\J_m V)}{Q,m}\norm{\I_m\xi}{Q,m},\\
      \Qm{(\rho M_0+M_1)\eta,\I_m\xi}_0
        &\leq \norm{\rho M_0+M_1}{}(\norm{V-\J_m V}{Q,m}+\norm{\J_m V-I \J_m V}{Q,m})\norm{\I_m\xi}{Q,m}\\
%         &\leq \norm{\rho M_0+M_1}{}\norm{V-I V}{Q,m}\norm{\I\xi}{Q,m},\\
      \Qm{A\eta,\I_m\xi}_0
        &\leq (\norm{AV-\J_m AV}{Q,m}+\norm{A(\J_m V-I \J_m V)}{Q,m})\norm{\I_m\xi}{Q,m}
%         &\leq \norm{A(V-I V)}{Q,m}\norm{\I\xi}{Q,m}\\
%       \Qm{\pt_t M_0\eta_\tau,\I\xi}
%         &= \Qm{\pt_t M_0(V-\J V),\I\xi}\\
%         &= \Qm{\pt_t M_0(V-\K V),\I\xi}
%          \leq \norm{\pt_t M_0(V-\K V)}{L^\infty(I_m)}\norm{\I\xi}{L^1(I_m)},\\
%       \Qm{\pt_t M_0\eta_h,\I\xi}
%         &= \Qm{\pt_t M_0(\J V-I\J V),\I\xi}\\
%         &\leq \norm{\pt_t M_0\J V-I\pt_t M_0\J V)}{L^\infty(I_m)}\norm{\I\xi}{L^1(I_m)},\\
%       \Qm{(\rho M_0+M_1)\eta,\I\xi}_0
%         &\leq \norm{\rho M_0+M_1}{}(\norm{V-\J V}{L^\infty(I_m)}+\norm{\J V-I \J V}{L^\infty(I_m)})\norm{\I\xi}{L^1(I_m)},\\
%       \Qm{A\eta,\I\xi}_0
%         &\leq (\norm{AV-\J AV}{L^\infty(I_m)}+\norm{A(\J V-I \J V)}{L^\infty(I_m)})\norm{\I\xi}{L^1(I_m)}.
%       \Qm{\pt_t M_0\eta,\I\xi}
%         &= \Qm{\pt_t M_0(V-\J V),\I\xi}\\
%         &= \Qm{\pt_t M_0(V-\K V),\I\xi}
%          \leq \norm{\pt_t M_0(V-\K V)}{L^\infty(I_m)}\norm{\I\xi}{L^1(I_m)},\\
%       \Qm{(\rho M_0+M_1)\eta,\I\xi}_0
%         &\leq \norm{\rho M_0+M_1}{}\norm{V-\J V}{L^\infty(I_m)}\norm{\I\xi}{L^1(I_m)},\\
%       \Qm{A\eta,\I\xi}_0
%         &\leq \norm{AV-\J AV}{L^\infty(I_m)}\norm{\I\xi}{L^1(I_m)}.
    \end{align*}
    where we used %repeatedly that $\J$ and $\Qm{\cdot,\cdot}$ use the same evaluation points $\tmi$
%     and 
    $\pt_t\J_m v(\tmi)=\pt_t\K_m v(\tmi)$.
    With $\norm{\I_m\xi}{Q,m}= \norm{\I_m\xi}{L^2(I_m)}$ and a Young inequality we obtain the result
    \begin{align*}
      \gamma\norm{\I\xi}{0}^2+&\sum_{m=1}^M\norm{M_0^{1/2}\jump{\I\xi}_{m-1}}H^2+\norm{M_0^{1/2}\xi(T)}H^2\\
         \leq C\big(& \norm{\pt_t M_0(V-\K V)}{Q}^2
                      +\norm{V-\J V}{Q}^2
                      +\norm{AV-\J AV}{Q}^2\\
                                     &+\norm{\pt_t M_0\J V-I\pt_t M_0\J V}{Q}^2
                                     +\norm{\J V-I \J V}{Q}^2
                                     +\norm{A(\J V-I\J V)}{Q}^2
                               \big).
    \end{align*}
    Together with $\norm{v}{Q}\leq CT^{1/2}\norm{v}{L^\infty([0,T])}$ and the optimal error estimates 
    for the interpolation operators in 
    space \eqref{eq:interspace} 
    and 
    time \eqref{eq:intertime} 
    we get 
    \begin{equation}\label{eq:pointbound}
      \gamma\norm{\I\xi}{0}^2+\sum_{m=1}^M\norm{M_0^{1/2}\jump{\I\xi}_{m-1}}H^2+\norm{M_0^{1/2}\xi(T)}H^2
        \leq C T(\tau^{2(q+2)}+h^{2k}).
    \end{equation}
    This bounds the interpolated discrete error. We get an estimate of the error itself by using \eqref{eq:J_jump_I},
    due to $\xi$ being continuous,
    and the mesh constraint $\tau_m/\tau_{m-1}\leq C$. Let us start with $m=1$. Then using an inverse inequality
    we have
    \[
      \norm{\xi}{L^2(I_1)}
        \leq \norm{\I_1\xi}{L^2(I_1)}+\norm{(\I_1\xi)(0^+)\theta_1}{L^2(I_1)}
        \leq \norm{\I_1\xi}{L^2(I_1)}+C_{inv}\norm{\I_1\xi}{L^2(I_1)}
    \]
    due to $\norm{\theta_1}{L^2(I_1)}\leq C\tau_1^{1/2}$. Similarly we obtain for $m\geq 2$
    \begin{align*}
      \norm{\xi}{L^2(I_m)}
        &\leq \norm{\I_m\xi}{L^2(I_m)}+\norm{(\I_m\xi)(t_{m-1}^+)\theta_m}{L^2(I_m)}+\norm{(\I_{m-1}\xi)(t_{m-1}^-)\theta_m}{L^2(I_m)}\\
        &\leq \norm{\I_m\xi}{L^2(I_m)}+C_{inv}\left(\norm{\I_m\xi}{L^2(I_m)}+\left( \frac{\tau_m}{\tau_{m-1}} \right)^{1/2}\norm{\I_{m-1}\xi}{L^2(I_{m-1})}\right).
    \end{align*}
    Thus, we have
    \begin{align*}
      \norm{\xi}{0}^2
        &\leq (1+C_{inv})^2\norm{\I_1\xi}{L^2(I_1)}^2+2\sum_{m=2}^M\left(\left( 1+C_{inv} \right)^2\norm{\I_m\xi}{L^2(I_m)}^2+C_{inv}^2\frac{\tau_m}{\tau_{m-1}}\norm{\I_{m-1}\xi}{L^2(I_{m-1})}^2 \right)\\
        &\leq C \norm{\I\xi}{0}^2.
    \end{align*}
    Together with the estimation of the interpolation error in time and space the proof is done.
  \end{proof}
  \begin{rem}\label{rem:U:post}
    For the original problem $U$, we obtain a post-processed approximation in $I_m$ as
    \[
      E_\rho\tilde V_h^\tau
        = E_\rho V_h^\tau - \jump{E_\rho V_h^\tau}_{m-1}^{x_0}\theta
    \]
    with an error estimate of
    \[
      \norm{E_\rho\tilde V^\tau_h-U}{0}\leq C\e^{\rho T}T^{1/2}(\tau^{q+2}+h^{k}).
    \]
%     Furthermore, that post-processed solution $E_\rho\tilde V_h^\tau$ is now energy-conserving
%     in the sense of Rem.~\ref{rem:energy:dG}.
  \end{rem}
  Tables~\ref{tab:prob1:post}
  \begin{table}[tbp]
    \caption{Convergence in $L^2$ of $V_h^\tau$ and $\tilde V_h^\tau$ to $V$, and $E_\rho \tilde V_h^\tau$ and $\tilde U_h^\tau$ to $U$ for Example 1 using various polynomial degrees and $\rho=2$\label{tab:prob1:post}}
    \begin{center}
%       \begin{tabular}{cc|cccccc}
%         $k/q$ & $M=N$ & \multicolumn{2}{c}{$\norm{V-V_h^\tau}{0}$} 
%                       & \multicolumn{2}{c}{$\norm{V-\tilde V_h^\tau}{0}$} 
%                       & \multicolumn{2}{c}{$\norm{U-E_\rho\tilde V_h^\tau}{0}$}\\
%         \hline
%         \multirow{3}{*}{3/1} 
% %        &   24 & 9.691e-04 &      &  1.269e-05 &      & 2.650e-05 &      \\
% %        &   48 & 2.429e-04 & 2.00 &  1.577e-06 & 3.01 & 3.154e-06 & 3.07 \\
% %        &   96 & 6.079e-05 & 2.00 &  1.971e-07 & 3.00 & 3.894e-07 & 3.02 \\
% %        &  192 & 1.521e-05 & 2.00 &  2.467e-08 & 3.00 & 4.854e-08 & 3.00 \\
%        & 192 & 1.521e-05 &      &  2.467e-08 &      & 4.854e-08 &      \\
%        & 384 & 3.803e-06 & 2.00 &  3.086e-09 & 3.00 & 6.065e-09 & 3.00 \\
%        & 768 & 9.508e-07 & 2.00 &  3.859e-10 & 3.00 & 7.581e-10 & 3.00 \\  
%         \hline
%         \multirow{3}{*}{4/2} 
% %        &  24 &  6.721e-06 &      &  2.604e-07 &      &  1.018e-06 &     \\
% %        &  48 &  8.4\frac{12}{}%Ce-07 & 3.00 &  1.624e-08 & 4.00 &  6.344e-08 & 4.00\\
% %        &  96 &  1.052e-07 & 3.00 &  1.015e-09 & 4.00 &  3.961e-09 & 4.00\\
% %        & 192 &  1.316e-08 & 3.00 &  6.342e-11 & 4.00 &  2.475e-10 & 4.00\\
%        & 192 &  1.316e-08 &      &  6.342e-11 &      &  2.475e-10 &     \\
%        & 384 &  1.645e-09 & 3.00 &  3.964e-12 & 4.00 &  1.547e-11 & 4.00\\
%        & 768 &  2.056e-10 & 3.00 &  2.621e-13 & 3.92 &  1.022e-12 & 3.92\\
%       \end{tabular}
      \begin{tabular}{cc|cccccccc}
        $k/q$ & $M=N$ & \multicolumn{2}{c}{$\norm{V-V_h^\tau}{0}$} 
                      & \multicolumn{2}{c}{$\norm{V-\tilde V_h^\tau}{0}$} 
                      & \multicolumn{2}{c}{$\norm{U-E_\rho\tilde V_h^\tau}{0}$}
                      & \multicolumn{2}{c}{$\norm{U-\tilde U_h^\tau}{0}$}\\
        \hline
        \multirow{3}{*}{3/1} 
%        &  24 & 9.691e-04 &      &  1.269e-05 &      & 2.650e-05 &      &  5.094e-06 &     \\
%        &  48 & 2.429e-04 & 2.00 &  1.577e-06 & 3.01 & 3.154e-06 & 3.07 &  5.961e-07 & 3.10\\
%        &  96 & 6.079e-05 & 2.00 &  1.971e-07 & 3.00 & 3.894e-07 & 3.02 &  7.350e-08 & 3.02\\
%        & 192 & 1.521e-05 & 2.00 &  2.467e-08 & 3.00 & 4.854e-08 & 3.00 &  9.174e-09 & 3.00\\
       & 192 & 1.521e-05 &      &  2.467e-08 &      & 4.854e-08 &      &  9.174e-09 &     \\
       & 384 & 3.803e-06 & 2.00 &  3.086e-09 & 3.00 & 6.065e-09 & 3.00 &  1.148e-09 & 3.00\\
       & 768 & 9.508e-07 & 2.00 &  3.859e-10 & 3.00 & 7.581e-10 & 3.00 &  1.436e-10 & 3.00\\  
        \hline
        \multirow{3}{*}{4/2}                                             
%        &  24 & 6.721e-06 &      &  2.604e-07 &      & 1.018e-06 &      &  2.527e-07 &     \\
%        &  48 & 8.412e-07 & 3.00 &  1.624e-08 & 4.00 & 6.344e-08 & 4.00 &  1.575e-08 & 4.00\\
%        &  96 & 1.052e-07 & 3.00 &  1.015e-09 & 4.00 & 3.961e-09 & 4.00 &  9.840e-10 & 4.00\\
%        & 192 & 1.316e-08 & 3.00 &  6.342e-11 & 4.00 & 2.475e-10 & 4.00 &  6.149e-11 & 4.00\\
       & 192 & 1.316e-08 &      &  6.342e-11 &      & 2.475e-10 &      &  6.149e-11 &     \\
       & 384 & 1.645e-09 & 3.00 &  3.964e-12 & 4.00 & 1.547e-11 & 4.00 &  3.851e-12 & 4.00\\
       & 768 & 2.056e-10 & 3.00 &  2.621e-13 & 3.92 & 1.022e-12 & 3.92 &  2.543e-13 & 3.92\\
      \end{tabular}
    \end{center}
  \end{table}
  and~\ref{tab:prob2:post}
  \begin{table}[tbp]
    \caption{Convergence in $L^2$ of $V_h^\tau$ and $\tilde V_h^\tau$ to $V$, and $E_\rho \tilde V_h^\tau$ and $\tilde U_h^\tau$ to $U$ for Example 2 using various polynomial degrees and $\rho=1$\label{tab:prob2:post}}
    \begin{center}
%       \begin{tabular}{cc|cccccc}
%         $k/q$ & $M=N$ & \multicolumn{2}{c}{$\norm{V-V_h^\tau}{0}$} 
%                       & \multicolumn{2}{c}{$\norm{V-\tilde V_h^\tau}{0}$} 
%                       & \multicolumn{2}{c}{$\norm{U-E_\rho\tilde V_h^\tau}{0}$}\\
%         \hline
%         \multirow{3}{*}{3/1} 
%       &   8 & 7.598e-04 &      &  3.865e-05 &      & 5.365e-05 &     \\
%       &  16 & 1.927e-04 & 1.98 &  4.865e-06 & 2.99 & 6.681e-06 & 3.01\\
%       &  16 & 1.927e-04 &      &  4.865e-06 &      & 6.681e-06 &     \\
%       &  32 & 4.851e-05 & 1.99 &  6.109e-07 & 2.99 & 8.372e-07 & 3.00\\
%       &  64 & 1.217e-05 & 2.00 &  7.655e-08 & 3.00 & 1.049e-07 & 3.00\\
%         \hline
%         \multirow{4}{*}{4/2} 
%       &   8 & 1.541e-05 &      &  6.796e-07 &      & 1.406e-06 &     \\
%       &  16 & 1.940e-06 & 2.99 &  4.239e-08 & 4.00 & 8.760e-08 & 4.00\\
%       &  16 & 1.940e-06 &      &  4.239e-08 &      & 8.760e-08 &     \\
%       &  32 & 2.433e-07 & 3.00 &  2.649e-09 & 4.00 & 5.472e-09 & 4.00\\
%       &  64 & 3.046e-08 & 3.00 &  2.195e-10 & 3.59 & 4.560e-10 & 3.58\\
%       \end{tabular}
      \begin{tabular}{cc|cccccccc}
        $k/q$ & $M=N$ & \multicolumn{2}{c}{$\norm{V-V_h^\tau}{0}$} 
                      & \multicolumn{2}{c}{$\norm{V-\tilde V_h^\tau}{0}$} 
                      & \multicolumn{2}{c}{$\norm{U-E_\rho\tilde V_h^\tau}{0}$}
                      & \multicolumn{2}{c}{$\norm{U-\tilde U_h^\tau}{0}$}\\
        \hline
        \multirow{3}{*}{3/1} 
%       &   8 & 7.598e-04 &      &  3.865e-05 &      & 5.365e-05 &      & 3.003e-05 &     \\
%       &  16 & 1.927e-04 & 1.98 &  4.865e-06 & 2.99 & 6.681e-06 & 3.01 & 3.764e-06 & 3.00\\
      &  16 & 1.927e-04 &      &  4.865e-06 &      & 6.681e-06 &      & 3.764e-06 &     \\
      &  32 & 4.851e-05 & 1.99 &  6.109e-07 & 2.99 & 8.372e-07 & 3.00 & 4.724e-07 & 2.99\\
      &  64 & 1.217e-05 & 2.00 &  7.655e-08 & 3.00 & 1.049e-07 & 3.00 & 5.920e-08 & 3.00\\
        \hline                                                        
        \multirow{4}{*}{4/2}                                           
%       &   8 & 1.541e-05 &      &  6.796e-07 &      & 1.406e-06 &      & 6.218e-07 &     \\
%       &  16 & 1.940e-06 & 2.99 &  4.239e-08 & 4.00 & 8.760e-08 & 4.00 & 3.874e-08 & 4.00\\
      &  16 & 1.940e-06 &      &  4.239e-08 &      & 8.760e-08 &      & 3.874e-08 &     \\
      &  32 & 2.433e-07 & 3.00 &  2.649e-09 & 4.00 & 5.472e-09 & 4.00 & 2.419e-09 & 4.00\\
      &  64 & 3.046e-08 & 3.00 &  2.195e-10 & 3.59 & 4.560e-10 & 3.58 & 1.740e-10 & 3.80\\
      \end{tabular}
    \end{center}
  \end{table}
  show the results of this post-processing procedure. Here we choose the polynomial degrees $q=k-2\geq 1$ in order to balance the 
  two error components. It is clear that the post-processed solutions $\tilde V_h^\tau$ and $E_\rho\tilde V_h^\tau$
  converge to $V$ and $U$, respectively with an improved order of $q+2=k$.
  Note that for $q=0$ the interpolation operator $\K$ in \eqref{eq:inter:K} is not defined and thus the 
  analysis presented in Theorem~\ref{thm:post} does not work. Furthermore, the numerical results do also not show 
  a better convergence rate for the post-processed solutions.
  
  In the last columns we also give numerical results for
  \[
    \tilde U_h^\tau:=U_h^\tau-\jump{U_h^\tau}_{m-1}^{x_0}\theta_\rho,
  \]
  which is the post-processing procedure applied to the solution $U_h^\tau$ of \eqref{eq:dG} using $\theta_\rho$ defined by the
  $\rho$-dependent interpolation points similar to $\theta$. 
  We also observe numerically an improved convergence order of $q+2$, but we do not have a theoretical basis for this result.

  As a final result we provide an $L^\infty$-error bound for the dG-C0 method~\eqref{eq:dG-C0} and thus for the post-processed solution.
  For this purpose, let us denote
  \[
%     \norm{\cdot}{M_0,\infty,m}^2:=\sup_{t\in I_m}\scp{M_0(\cdot)(t),(\cdot)(t)}
    \norm{\cdot}{M_0,\infty,m}^2:=\sup_{t\in I_m}\norm{M_0^{1/2}(\cdot)(t)}{H}^2
    \qmbox{and}
%     \norm{\cdot}{M_0,\infty}^2:=\sup_{t\in[0,T]}\scp{M_0(\cdot)(t),(\cdot)(t)}.
    \norm{\cdot}{M_0,\infty}^2:=\sup_{t\in[0,T]}\norm{M_0^{1/2}(\cdot)(t)}{H}^2.
  \]
  We follow the analysis of \cite{AM04}, see also \cite{VR13,FrTW16}.
  \begin{lem}\label{lem:AK04}
    Let $p\in\PS_{q+1}([0,1])$ and $\tilde p\in\P_q([0,1])$, where $\tilde p=\I(pm^{-1})$ with the Gauß-Radau
    interpolation operator $\I$ on $[0,1]$ and $m:=t\mapsto t$. Then it holds
    \[
      \scp[]{p',\tilde p}%=\Qm[]{p',\tilde p}
      =\frac{1}{2}\Qm[]{\tilde p,\tilde p}+\frac{1}{2}p(1)^2-p(0)\tilde p(0),
    \]
    where $\Qm[]{\cdot}$ is the Gauß-Radau quadrature rule on $[0,1]$.
  \end{lem}
  \begin{proof}
    The proof follows the same lines as \cite[Lemma 2.1]{AM04}, exploiting the interaction of the Gauß-Radau interpolation and quadrature.
    We will therefore skip it.
  \end{proof}
  As a consequence of Lemma~\ref{lem:AK04} we have
  \begin{align}
    \scp[]{p',\tilde p}
      &\geq \frac{1}{2}\Qm[]{\tilde p,\tilde p}-p(0)\tilde p(0)
       \geq \frac{1}{2}\Qm[]{p,p}-p(0)\tilde p(0)\notag\\
      &= \frac{1}{2}\Qm[]{\I p,\I p}-p(0)\tilde p(0)
       = \frac{1}{2}\norm{\I p}{L^2([0,1])}^2-p(0)\tilde p(0).\label{eq:AK04}
  \end{align}

  \begin{thm}
    Under the conditions of Theorem~\ref{thm:post} we have
    \[
%       \sup_{t\in[0,T]}\scp{M_0(V-\tilde V_h^\tau)(t),(V-V_h^\tau)(t)}
      \norm{V-\tilde V_h^\tau}{M_0,\infty}
        \leq CT^{1/2}(\tau^{q+2}+h^{k}).
    \]
    Furthermore, we also have
    \[
%       \sup_{t\in[0,T]}\scp{M_0(U-E_\rho\tilde V_h^\tau)(t),(V-V_h^\tau)(t)}
      \norm{U-E_\rho\tilde V_h^\tau}{M_0,\infty}
        \leq C\e^{\rho T}T^{1/2}(\tau^{q+2}+h^{k}).
    \]
  \end{thm}
  \begin{proof}
    Despite its similarity with the proof of \cite[Theorem 3.12]{FrTW16}, we provide it here.
    The biggest change is, that we provide an estimate for $\I\xi$ first.
    
    Let $p=\sqrt{M_0}\xi\in\V$, $\phi:=t\mapsto \frac{\tau_m}{t-t_{m-1}}\xi(t)$ and $\tilde p=\sqrt{M_0}\I\phi\in\U$.
    Then \eqref{eq:AK04} after rescaling to $I_m$ gives 
    \[
      \scp[m]{\pt_t M_0\xi,2\I_m\phi}
        \geq \frac{1}{\tau_m}\scp[m]{M_0\I_m\xi,\I_m\xi}-2\scp{M_0\xi(t_{m-1}), \I_m\phi(t_{m-1})}.
    \]
    With an inverse inequality we obtain
    \begin{align*}
      \norm{\I_m\xi}{M_0,\infty,m}^2
        &=\sup_{t\in I_m}\scp{M_0\I_m\xi,\I_m\xi}
         \leq \frac{C}{\tau_m}\scp[m]{M_0\I_m\xi,\I_m\xi}\\
        &\leq C(\scp[m]{\pt_t M_0\xi,2\I_m\phi}+2\scp{M_0\xi(t_{m-1}), \I_m\phi(t_{m-1})}).
    \end{align*}
    To use the error equation~\eqref{eq:erroreq}, we change the right-hand side to the full operator 
    of the differential equation. With
    \begin{align*}
      \scp[m]{\pt_t M_0\xi,2\I_m\phi}
        &= \Qm{\pt_t M_0\xi,2\I_m\phi},\\
      \Qm{A\xi,2\I_m\phi}
        &= \frac{\tau_m}{2}\sum_{i=0}^qw_i^m\frac{2\tau_m}{(\tmi-t_{m-1})\tmi}\scpr{A\xi(\tmi),\xi(\tmi)}=0
    \end{align*}
    it then follows
    \begin{align*}
      \norm{\I_m\xi}{M_0,\infty,m}^2
%       \sup_{t\in I_m}&\scp{M_0\I_m\xi(t),\I_m\xi(t)}\\
        &\leq \Qm{(\pt_t M_0+(\rho M_0+M_1)+A)\xi,2\I_m\phi} +2\scp{M_0\xi(t_{m-1}), \I_m\phi(t_{m-1})}\\&\hspace*{2cm}-\Qm{(\rho M_0+M_1)\xi,2\I_m\phi}\\
        &\leq \Qm{(\pt_t M_0+(\rho M_0+M_1)+A)\eta,2\I_m\phi}+2\scp{M_0\xi(t_{m-1}), \I_m\phi(t_{m-1})}\\&\hspace*{2cm}-\Qm{(\rho M_0+M_1)\xi,2\I_m\phi}.
    \end{align*}
    If we apply the discrete Cauchy-Schwarz inequality \eqref{eq:discr:CSU} to the quadrature terms, we obtain the factors $\Qm{\I_m\phi,\I_m\phi}$.
    For these, by definition of $\I_m$ and $\phi$, we have
    \[
      \Qm{\I_m\phi,\I_m\phi}
        =\Qm{\phi,\phi}
        \leq\frac{1}{c^2}\Qm{\xi,\xi}
        = \frac{1}{c^2}\Qm{\I_m\xi,\I_m\xi},
    \]
    where $c$ depends on $\frac{\tmi-t_{m-1}}{\tau_m}$, but is bounded from below, see \cite{Waurick16}.
    Now we can use the results of Theorem~\ref{thm:post} and its proof to first bound the quadrature terms:
    \begin{align*}
      \Qm{(\pt_t M_0+(\rho M_0+M_1)+A)\xi,2\I_m\phi}
        &\leq CT^{1/2}(\tau^{q+2}+h^k)\Qm{\I_m\phi,\I_m\phi}^{1/2}\\
        &\leq CT(\tau^{2(q+2)}+h^{2k}),\\
      \Qm{(\rho M_0+M_1)\xi,2\I_m\phi}
        &\leq C\norm{\rho M_0+M_1}{}\Qm{\xi,\xi}^{1/2}\Qm{\I_m\phi,\I_m\phi}^{1/2}\\
        &\leq C T(\tau^{2(q+2)}+h^{2k}).
    \end{align*}
    Finally, with a Young's inequality and the estimate of Theorem~\ref{thm:post} at the discrete times $t_{m-1}$ (shown similarly to the estimate at $T$ in \eqref{eq:pointbound})
    \begin{align*}
      \scp{M_0\xi(t_{m-1}), \I_m\phi(t_{m-1})}
%         &\leq \frac{1}{c}\scp{M_0\xi(t_{m-1}),\xi(t_{m-1})}^{1/2}\sup_{t\in I_m}\scp{M_0\I_m\xi(t),\I_m\xi(t)}^{1/2}\\
        &\leq \frac{1}{c}\norm{M_0^{1/2}\xi(t_{m-1})}{H}\sup_{t\in I_m}\norm{M_0^{1/2}\I_m\xi(t)}{H}\\
%         &\leq \frac{1}{c^2}T(\tau^{2(q+2)}+h^{2k})+\frac{1}{4}\sup_{t\in I_m}\scp{M_0\I_m\xi(t),\I_m\xi(t)}.
        &\leq \frac{1}{c^2}T(\tau^{2(q+2)}+h^{2k})+\frac{1}{4}\norm{\I_m\xi}{M_0,\infty,m}^2.
    \end{align*}
    Combining these estimates it follows
    \[
      \norm{\I_m\xi}{M_0,\infty,m}^2
%       \sup_{t\in I_m}\scp{M_0\I_m\xi(t),\I_m\xi(t)}
        \leq C T(\tau^{2(q+2)}+h^{2k}).
    \]
    With $\xi=\I_m\xi-\jump{\I\xi}_{m-1}^0\theta_m$ on $I_m$ we obtain
    \begin{align*}
      \norm{\xi}{M_0,\infty,m}^2
%       \sup_{I_m}&\scp{M_0\xi(t),\xi(t)}\\
%       &\leq 2\left( \sup_{t\in I_m}\scp{M_0\I\xi(t),\I\xi(t)}+\sup_{t\in I_m}\scp{M_0\jump{\I\xi}_{m-1}^0\theta_m(t),\jump{\I\xi}_{m-1}^0\theta_m(t)} \right)\\
%       &\leq 2\bigg( \sup_{t\in I_m}\scp{M_0\I_m\xi(t),\I_m\xi(t)}+\\
%         &\hspace{1.5cm}2\norm{\theta_m}{L^\infty(I_m)}^2\left(\sup_{t\in I_m}\scp{M_0\I_m\xi(t),\I_m\xi(t)}+\sup_{t\in I_{m-1}}\scp{M_0\I_{m-1}\xi(t),\I_{m-1}\xi(t)}\right) \bigg).
      &\leq 2\bigg( \norm{\I_m\xi}{M_0,\infty,m}^2+2\norm{\theta_m}{L^\infty(I_m)}^2\left(\norm{\I_m\xi}{M_0,\infty,m}^2+\norm{\I_{m-1}\xi}{M_0,\infty,m-1}^2\right) \bigg).
    \end{align*}
    It then follows
    \[
      \norm{\xi}{M_0,\infty}^2
%       \sup_{t\in[0,T]}\scp{M_0\xi(t),\xi(t)}
%         \leq 2(1+2\norm{\theta}{L^{\infty}}^2)\sup_{t\in[0,T]}\scp{M_0\I\xi(t),\I\xi(t)}
        \leq 2(1+2\norm{\theta}{L^{\infty}}^2)\norm{\I\xi}{M_0,\infty}^2
        \leq CT(\tau^{2(q+2)}+h^{2k}).
    \]
    Together with the interpolation error estimates, this completes the proof of the first part.
    The second result simply follows as in Remark~\ref{rem:U:post}. 
  \end{proof}

  Tables~\ref{tab:prob1:post:inf}
  \begin{table}[tbp]
    \caption{Convergence in $L^\infty$ of $V_h^\tau$ and $\tilde V_h^\tau$ to $V$, and $E_\rho \tilde V_h^\tau$ and $\tilde U_h^\tau$ to $U$ for Example 1 using various polynomial degrees and $\rho=2$\label{tab:prob1:post:inf}}
    \begin{center}
      \begin{tabular}{cc|cccccccc}
        $k/q$ & $M=N$ & \multicolumn{2}{c}{$\norm{V-V_h^\tau}{M_0,\infty}$} 
                      & \multicolumn{2}{c}{$\norm{V-\tilde V_h^\tau}{M_0,\infty}$} 
                      & \multicolumn{2}{c}{$\norm{U-E_\rho\tilde V_h^\tau}{M_0,\infty}$}
                      & \multicolumn{2}{c}{$\norm{U-\tilde U_h^\tau}{M_0,\infty}$}\\
        \hline
        \multirow{3}{*}{3/1} 
%       &   24 & 1.808e-03 &      &  2.771e-05 &      &  2.771e-05 &       & 5.377e-05 &     \\
%       &   48 & 4.677e-04 & 1.95 &  3.614e-06 & 2.94 &  3.614e-06 & 2.94  & 6.466e-06 & 3.06\\
%       &   96 & 1.189e-04 & 1.98 &  4.614e-07 & 2.97 &  4.614e-07 & 2.97  & 8.023e-07 & 3.01\\
%       &  192 & 2.999e-05 & 1.99 &  5.830e-08 & 2.98 &  5.830e-08 & 2.98  & 1.002e-07 & 3.00\\
      &  192 & 2.999e-05 &      &  5.830e-08 &      &  5.830e-08 &       & 1.002e-07 &     \\
      &  384 & 7.530e-06 & 1.99 &  7.326e-09 & 2.99 &  7.326e-09 & 2.99  & 1.253e-08 & 3.00\\
      &  768 & 1.887e-06 & 2.00 &  9.183e-10 & 3.00 &  9.183e-10 & 3.00  & 1.567e-09 & 3.00\\  
        \hline
        \multirow{3}{*}{4/2}   
%       &   24 & 1.934e-05 &      &  2.779e-07 &      &  2.779e-07 &      &  1.889e-06 &     \\
%       &   48 & 2.490e-06 & 2.96 &  1.722e-08 & 4.01 &  1.722e-08 & 4.01 &  1.173e-07 & 4.01\\
%       &   96 & 3.159e-07 & 2.98 &  1.075e-09 & 4.00 &  1.075e-09 & 4.00 &  7.316e-09 & 4.00\\
%       &  192 & 3.979e-08 & 2.99 &  6.713e-11 & 4.00 &  6.713e-11 & 4.00 &  4.570e-10 & 4.00\\
      &  192 & 3.979e-08 &      &  6.713e-11 &      &  6.713e-11 &      &  4.570e-10 &     \\
      &  384 & 4.992e-09 & 2.99 &  4.196e-12 & 4.00 &  4.196e-12 & 4.00 &  2.858e-11 & 4.00\\
      &  768 & 6.252e-10 & 3.00 &  2.674e-13 & 3.97 &  2.674e-13 & 3.97 &  1.838e-12 & 3.96
      \end{tabular}
    \end{center}
  \end{table}
  and~\ref{tab:prob2:post:inf}
  \begin{table}[tbp]
    \caption{Convergence in $L^\infty$ of $V_h^\tau$ and $\tilde V_h^\tau$ to $V$, and $E_\rho \tilde V_h^\tau$ and $\tilde U_h^\tau$ to $U$ for Example 2 using various polynomial degrees and $\rho=1$\label{tab:prob2:post:inf}}
    \begin{center}
      \begin{tabular}{cc|cccccccc}
        $k/q$ & $M=N$ & \multicolumn{2}{c}{$\norm{V-V_h^\tau}{M_0,\infty}$} 
                      & \multicolumn{2}{c}{$\norm{V-\tilde V_h^\tau}{M_0,\infty}$} 
                      & \multicolumn{2}{c}{$\norm{U-E_\rho\tilde V_h^\tau}{M_0,\infty}$}
                      & \multicolumn{2}{c}{$\norm{U-\tilde U_h^\tau}{M_0,\infty}$}\\
        \hline
        \multirow{3}{*}{3/1} 
%      &  4 & 9.113e-03 &      &  6.068e-04 &       & 6.068e-04 &      &  7.375e-04 &     \\
%      &  8 & 2.631e-03 & 1.79 &  8.762e-05 & 2.79  & 8.762e-05 & 2.79 &  8.884e-05 & 3.05\\
%      & 16 & 7.075e-04 & 1.90 &  1.177e-05 & 2.90  & 1.177e-05 & 2.90 &  1.109e-05 & 3.00\\
     & 16 & 7.075e-04 &      &  1.177e-05 &       & 1.177e-05 &      &  1.109e-05 &     \\
     & 32 & 1.834e-04 & 1.95 &  1.525e-06 & 2.95  & 1.525e-06 & 2.95 &  1.391e-06 & 3.00\\
     & 64 & 4.671e-05 & 1.97 &  1.940e-07 & 2.97  & 1.940e-07 & 2.97 &  1.742e-07 & 3.00\\
        \hline
        \multirow{3}{*}{4/2}                                             
%       &  4 & 4.425e-04 &      &  1.274e-05 &      &  1.274e-05 &      &  3.472e-05 &     \\
%       &  8 & 6.167e-05 & 2.84 &  7.845e-07 & 4.02 &  7.845e-07 & 4.02 &  2.206e-06 & 3.98\\
%       & 16 & 8.141e-06 & 2.92 &  4.885e-08 & 4.01 &  4.885e-08 & 4.01 &  1.395e-07 & 3.98\\
      & 16 & 8.141e-06 &      &  4.885e-08 &      &  4.885e-08 &      &  1.395e-07 &     \\
      & 32 & 1.046e-06 & 2.96 &  3.075e-09 & 3.99 &  3.075e-09 & 3.99 &  8.791e-09 & 3.99\\
      & 64 & 1.325e-07 & 2.98 &  1.952e-10 & 3.98 &  1.952e-10 & 3.98 &  5.532e-10 & 3.99
      \end{tabular}
    \end{center}
  \end{table}
  show the pointwise convergence results for our two examples. Again, the theoretical results are supported by our experiments, 
  and the order $q+2$ is achieved for the post-processed solution.
  
\section*{Acknowledgements}
  The author would like to thank Simon Becher and Gunar Matthies for the helpful discussions on
  the post-processing of dG-methods.
  
  \bibliographystyle{plain}
  \bibliography{lit}
\end{document}